\documentclass[10pt]{amsart}
\usepackage{setspace,kantlipsum}
\usepackage[utf8]{inputenc}
\usepackage{graphicx}
\graphicspath{ {./images/} }
\usepackage{amssymb}
\usepackage[foot]{amsaddr}
\usepackage{bbm}
\usepackage[all,arc]{xy}
\usepackage{enumerate}
\usepackage{mathrsfs}
\usepackage{pst-node}
\usepackage[margin=1in]{geometry} 
\usepackage{amsmath,amsthm}
\usepackage{mathtools}
\usepackage[T1]{fontenc}
\usepackage[sc,osf]{mathpazo}
\usepackage[euler-digits,small]{eulervm}

\usepackage[hyphens]{url}
\usepackage{hyperref}
\usepackage[capitalise,nameinlink]{cleveref}
\usepackage{doi}
\newcommand*\bb{\mathbb}

\newcommand *\w{^\wedge}

\newcommand*\de{\partial}

\newcommand*\Ren{\textup{\textbf{R}}}
\newcommand*\FenR{\textup{\textbf{F}}_{R}}
\newcommand*\FenG{\textup{\textbf{F}}_{\mathfrak{G}}}
\newcommand*\lone{\textbf{L}^{\textup{1}}}
\newcommand*\divG{{\textbf{div}^\alpha_{\mathfrak{G}}}}
\newcommand*\bvG{\textbf{BV}_{\mathfrak{G}}^{\alpha}}
\newcommand*\tvG{\textbf{TV}_{\mathfrak{G}}^{\alpha}}

\newcommand*\lwbvGz{\textbf{L}_{\textbf{w}}^1\left(0,T;\bvG(\Omega;u_0)\right)}
\newcommand*\lwbvGzvz{\textbf{L}_{\textbf{w}}^1\left(0,T;\bvG(\Omega;v_0)\right)}
\newcommand*\bbR{\mathbb{R}}
\newcommand*\bbr{{\mathbb{R}^{\textup{N}}}}
\newcommand*\bvr{\textup{BV}_{\textup{R}}^{\alpha}}
\newcommand*\tvr{\textup{TV}_{\textup{R}}^{\alpha}}
\newcommand*\lwbv{L_{\textup{w}}^1\left(0,T;\bvr(\Omega)\right)}
\newcommand*\divr{\nabla_{\textup{R}}^{\alpha}}

\newtheorem{thm}{Theorem}[section]
\theoremstyle{definition}
\newtheorem{defn}[thm]{Definition}

\newtheorem{cor}[thm]{Corollary}
\newtheorem{prop}[thm]{Proposition}
\newtheorem{lem}[thm]{Lemma}
\newtheorem{rem}[thm]{Remark}
\usepackage{xcolor}
\hypersetup{
    colorlinks,
    linkcolor={red!50!black},
    citecolor={red!50!black},
    urlcolor={blue!80!black}
}

\DeclarePairedDelimiterX{\inp}[2]{\langle}{\rangle}{#1, #2}

\makeatletter
\def\@tocline#1#2#3#4#5#6#7{\relax
  \ifnum #1>\c@tocdepth 
  \else
    \par \addpenalty\@secpenalty\addvspace{#2}%
    \begingroup \hyphenpenalty\@M
    \@ifempty{#4}{%
      \@tempdima\csname r@tocindent\number#1\endcsname\relax
    }{%
      \@tempdima#4\relax
    }%
    \parindent\z@ \leftskip#3\relax \advance\leftskip\@tempdima\relax
    \rightskip\@pnumwidth plus4em \parfillskip-\@pnumwidth
    #5\leavevmode\hskip-\@tempdima
      \ifcase #1
       \or\or \hskip 1em \or \hskip 2em \else \hskip 3em \fi%
      #6\nobreak\relax
    \dotfill\hbox to\@pnumwidth{\@tocpagenum{#7}}\par
    \nobreak
    \endgroup
  \fi}
\makeatother
\usepackage{hyperref}
\numberwithin{equation}{section}
\usepackage{epigraph}

\title{A variational approach to nonlocal image restoration flows}
\author{Harsh Prasad $^{\ast}$}
\author{Vivek Tewary$^{\dagger}$}
\address{$^{\ast}$Fakultät für Mathematik, Universität Bielefeld, Postfach 100131, D-33501 Bielefeld, Germany.}
\address{$^{\dagger}$School of Interwoven Arts and Sciences, Krea University, Sri City, Andhra Pradesh, 517646, India.}
\email{$^{\ast}$hprasad@math.uni-bielefeld.de}
\email{$^{\dagger}$vivek.tewary@krea.edu.in}
\thanks{$^{\ast}$Support from Deutsche Forschungsgemeinschaft (DFG, German Research Foundation) – Project-ID 317210226 – SFB 1283 is gratefully acknowledged}
\date{October 2024}
\begin{document}
\begin{abstract}
 We prove existence, uniqueness and initial time regularity for variational solutions to nonlocal total variation flows associated with image denoising and deblurring. In particular, we prove existence of parabolic minimisers $u$, that is,
 $$\int_0^T\int_\Omega u\de_t\phi\,dx + \textbf{F}(u(t))\,dt\leq \int_0^T \textbf{F}(u+\phi)(t)\,dt,$$ for $\phi\in C^\infty_c(\Omega\times (0,T))$. The prototypical functional $\textbf{F}(u)$ is $\textbf{F}(u)=\textbf{TV}^{\alpha}_{\cdot}(u)+\frac{\kappa}{\zeta}\int_\Omega|u(x)-u_0(x)|^\zeta\,dx$ for $\zeta\geq 1$. Here $\textbf{TV}^{\alpha}_{\cdot}$ is a fractional total variation of either the Riesz or the Gagliardo type and the second term is a regression term. These models are based on different definitions of fractional $\textbf{BV}$ spaces that have been proposed in the literature. The notion of solution is completely variational and based on the weighted dissipation method. We demonstrate existence without smoothness assumptions on the domain and exhibit uniqueness without using strict convexity. We can also deal with fairly general fidelity or regression terms in the model. Furthermore, the method also provides a novel route to constructing solutions of the parabolic fractional $1$-Laplace equation. 
\end{abstract}
\maketitle
\tableofcontents
\section{Introduction}
We are interested in studying certain models for nonlocal image restoration flows including both the denoising and the deblurring flows from a variational standpoint. To explain our interest in the models under consideration, let us begin with two classical variants for image denoising, namely the Rudin-Osher-Fatemi (ROF) model \cite{Rudin1992} and the Chan-Esedoglu (CE) model \cite{Chan2005}.  
\subsection{ROF, CE and generalised regression models} Given a corrupted or noisy monochromatic image $u_0$, we want to restore the original image. To do this, the ROF model proposes to minimise the functional
\[
F_{ROF}(v) = \int_{\Omega} |\nabla v| \,dx + \Lambda\int_{\Omega}|v-u_0|^2 \,dx
\]
where $\Omega \subset \bbr$ is a bounded domain and $\Lambda > 0$. The first term is the regularisation term whereas the second is the regression term which ensures that we do not stray too far away from the original image. An important limitation of the ROF model is that it does not have a large enough class of noise-free images which it leaves invariant - it is not contrast invariant \cite[Section 1]{Chan2005}. On the other hand, the CE model proceeds by minimising 
\[
F_{CE}(v) = \int_{\Omega} |\nabla v| \,dx + \Lambda\int_{\Omega}|v-u_0| \,dx.
\]
and is, in fact, contrast invariant. Of course, we lose uniqueness in the CE model - however, this may not necessarily be a disadvantage \cite[Sections 6 and 7]{Chan2005}. In fact, besides the $L^2$ and $L^1$ norms, one may also consider logistic or quantile or robust regression terms \cite{Osher2004}. They can all be modelled by the considering
\[
\int_{\Omega} \mathbf{R}(x,u(x)) \,dx
\]
for any function 
\[
\mathbf{R}: \Omega \times \bb{R} \rightarrow [0,\infty)
\]
such that
\begin{itemize}
    \item[1.] $\mathbf{R}$ is a Carathéodory integrand and
    \item[2.]  $\mathbf{R}$ is convex in the second variable for a.e. $x \in \Omega$.
\end{itemize}
\subsection{Generalised deblurring models} In deblurring, one is interested in recovering a sharp image from a blurry observation $u_0$. Since it is an inverse problem, it is important to choose an appropriate regularisation. Here we focus on the Tikhonov regularisation \cite{Tihonov1963, Tihonov1963a, Tikhonov1963} and assume that we have a linear blurring operator 
\[
K : L^1(\Omega) \rightarrow L^2(\Omega)
\]
which is bounded, injective and satisfies the DC condition namely 
\[
K[1] = 1
\]
Recall that we work with bounded domains. DC stands for direct-current and the condition says that the adjoint $K^*$ preserves means. Examples of linear blur include motion blurs, defocused blurs and atmospheric blurs \cite[Chapter 5]{Chan2005a}. To solve the deblurring inverse problem via  Tikhonov regularisation one minimises the following functional:
\[
F_{K}(v) = \int_{\Omega} |\nabla v| \,dx + \Lambda\int_{\Omega}|K[v]-u_0|^2 \,dx.
\]
The model was studied in \cite{Acar1994, Chambolle1997}. 
\subsection{Nonlocal models} From an abstract point of view, the above models differ in the penalisation term and keep the same regularisation term. From such a vantage, it is therefore not unexpected to ask what happens if we also change how we regularise. In fact, doing so is not just a mathematical curiosity. In the models above where we regularise by minimising the gradient, we are implicitly assuming some (spatial) regularity of the final image. For example, we are assuming $BV$ regularity above. This may lead to a loss of finer details in the image such as different textures which will be smoothed out. Therefore, it is desirable to lower the regularity requirements. One of the ways of doing this is to introduce nonlocal functionals to replace the $BV$ term i.e. the regularisation term. Many such models have indeed been proposed in literature and they seem to have better performance than the local model based on the $BV$ term \cite{Buades2005, Gilboa2007, Gilboa2009, Kindermann2005, Mahmoudi2005}.
\subsection{Nonlocal BV} In contrast to the nonlocal models referenced above, we are interested in studying regularisation by replacing the $BV$ term by nonlocal variants of $BV$. There are at least two 'natural' definitions of what constitutes nonlocal $BV$ in the literature. Following \cite{Antil2024} we call these the Gagliardo $BV$ space and the Riesz $BV$ space (see \cref{sec:gag} and \cref{sec:riesz}.) Let us explain the properties of these different nonlocal $BV$ spaces. The two $BV$ spaces use two different notions of the gradient and hence the divergence operator. For the Riesz case \cite{Horvath1959}, the definition of the gradient arises naturally from symmetry and continuity requirements - namely invariance under translations and rotations, homogeneity under dilation and continuity for example in Schwartz space. In fact, it is the unique such operator \cite{Silhavy2019}. The Riesz $BV$ have been studied extensively in recent years - see for example \cite{Comi2022, Comi2023, Comi2022a, Comi2023a, Schikorra2015, Schikorra2017, Schikorra2017a}. On the other hand, the Gagliardo $BV$ space has shown up in as varied instances as fractional perimeters \cite{Caffarelli2010}, peridynamics \cite{DU2013}, Dirichlet forms \cite{Hinz2015} and harmonic analysis \cite{Mazowiecka2018}. In fact it can be identified with the fractional Sobolov space $W^{s,1}$ \cite[Theorem 3.4]{Antil2024}. 
\subsection{Image restoration via nonlocal \texorpdfstring{$BV$}{BV} flows} In this article we study image restoration flows driven by nonlocal $BV$ regularisation terms - both deblurring and denoising - and establish existence, uniqueness, initial (time) regularity and energy estimates for such flows. In the local case, i.e., when working with the usual $BV$ space, such results were first established in \cite{Boegelein2015} wherein the existence is proved via minimising certain convex functionals. The approach in \cite{Boegelein2015} is quite flexible and lets one deal with fairly general regression terms and linear blurs. We adapt their approach to the nonlocal setup. In the time-independent setup, the image denoising problem was studied in \cite{Antil2024} for $L^p$ regression terms for $p \in [1,p_{\infty})$ with $p_{\infty} < \infty$ in bounded Lipschitz domains. The reason there is a restriction on the upper range of $p$ is due to a lack of compactness. The problem with compactness has also been noted in \cite{Antil2024a}. Existence in the time independent setup for denoising functionals in the Riesz case under growth restrictions have also been obtained in \cite{Schoenberger2024} whereas the denoising problem with the Gagliardo $BV$ space with $L^1$ and $L^2$ regression terms were studied in \cite{Bessas2022, Novaga2022}. Following \cite{Boegelein2015} we are able to remove the growth restrictions on the regression terms and are in fact able to deal with fairly general Carathéodory integrands without any growth assumptions. Again as in \cite{Boegelein2015} we are also able to deal with fairly general linear blurs. Furthermore, we can work in arbitrary domains in the Gagliardo case because the "Cauchy data" is already prescribed for us so there is no extension procedure involved as in \cite[Proposition 3.10]{Antil2024}  (for the Riesz case, the existence result in \cite{Antil2024} already holds for arbitrary domains.) Finally, as far as uniqueness is concerned, we are able to work without strict convexity - in particular, as compared to the time-independent setup of \cite{Antil2024}, we are able to get uniqueness for $L^1$ regression terms. This is made possible due to the parabolic nature of the technique (see \cref{sec:uni}.) Let us note that in \cite{Bessas2022}, uniqueness for the Gagliardo case with $L^1$ regression terms is obtained under further restrictions on the regularity of the level sets of $u_0$ for high values of $\Lambda$ whereas for low values of $\Lambda$ there are restrictions on the support of $u_0$. The two key ideas we use are the Weighted Inertia-Dissipation-Energy (WIDE) technique  and an $L^1$ compactness lemma. The compactness lemma is similar to \cite[Lemma 8.1]{Boegelein2015} and \cite{Simon1986} where we were able to use results already established in \cite{Antil2024} along with Banach valued Arzela-Ascoli \cite[Chapter 3, Theorem 3.1]{Lang1993} to prove \cref{lem:com} and \cref{lem:com2}. The WIDE approach is a general principle that lets one use elliptic-in-time regularisation to study various nonlinear evolution problems. It was first applied in the setup where the evolution is governed by anomalous diffusion operators such as the fractional p-Laplacian in \cite{Prasad2023}. We refer the reader to the article \cite{Stefanelli2024} for an extensive review of this technique.
\section{Gradient Flows to nonlocal \texorpdfstring{$BV$}{BV} based image restoration}
\subsection{ \texorpdfstring{$\tvr$}{tvr}-Deblurring Flow}
Suppose that 
\begin{align}\label{deblurfunc}
\begin{cases}
&K:L^1(\Omega) \rightarrow L^2(\Omega) \text{ is a bounded, injective function}\\
&\text{and satisfies } K[1]=1.
\end{cases}
\end{align}
The condition $K[1] = 1$ is called the DC condition. For $u_0 \in L^2(\Omega)$ and penalisation $\kappa \geq 1$, the functional is
\[
\FenR(u) = \tvr(u) + \frac{\kappa}{2}\int_{\Omega}|K[u]-u_0|^2\,dx.
\]
We will assume that
\[
u_0 \in L^2(\Omega) \cap \bvr(\Omega).
\]
which implies that
\begin{align}\label{finiteenergydeblur}
\FenR(u_0) < +\infty.
\end{align}
\begin{defn}\label{def:solr}
    Denote by $\Omega_T$ the set $\Omega\times [0,T]$. A function $u: \Omega_T \rightarrow \mathbb{R}$ for $0<T<\infty$ is called a variational solution in $\Omega_T$ for the $\tvr$ deblurring flow provided
    \[
    u \in \lwbv\cap C^0([0,T];L^2(\Omega))
    \]
    and the following variational inequality holds:
    \begin{gather}
    \label{eq:sol}
    \begin{aligned}
        \int_0^T \FenR(u(t)) \,dt \leq& \int_0^T\int_{\Omega}\partial_t v(v-u) \,dx \,dt + \int_0^T \FenR(v(t))\,dt\\
        &-\frac{1}{2}\|(v-u)(T)\|_{L^2(\Omega)}^2+\frac{1}{2}\|v(0)-u_0\|_{L^2(\Omega)}^2
    \end{aligned}
    \end{gather}
    for any $v\in \lwbv$ with $\partial_tv\in L^2(\Omega_T)$ and $v(0) \in L^2(\Omega)$. If $u$ is a variational solution in $\Omega_T$ for every $T>0$, we say that $u$ is a global variational solution. 
\end{defn}
\begin{rem}
    Since $v(x,t) := u_0(x)$ is a valid comparison map, it follows that variational solutions have finite energy because $u_0$ has finite energy.
\end{rem}
\begin{thm}[Existence]\label{existence1}
    Suppose that $K$ is a linear blur as in \cref{deblurfunc} and $u_0$ satisfies \cref{finiteenergydeblur}. Then there exists a global variational solution to the $\tvr$-deblurring flow as defined in \cref{def:solr}.   
\end{thm}
\begin{thm}[Uniqueness]\label{uniqueness1}
    Suppose that $K$ is a linear blur as in \cref{deblurfunc} and $u_0$ satisfies \cref{finiteenergydeblur}. Then there is at most one global variational solution to the $\tvr$-deblurring flow as defined in \cref{def:solr}. 
\end{thm}
\begin{thm}[Time Regularity]\label{timeregularity1}
    Suppose that $K$ is a linear blur as in \cref{deblurfunc} and $u_0$ satisfies \cref{finiteenergydeblur}. Then any global variational solution to the $\tvr$-deblurring flow as defined in \cref{def:solr} satisfies:
    \begin{itemize}
        \item[(i)] \[u \in C^{0,\frac{1}{2}}([0,T];L^2(\Omega)) \quad \mbox{ for every } \quad T>0.\]
        \item[(ii)] \[
        \int_0^T\int_{\Omega}|\partial_t u|^2 \,dx\,dS \leq \FenR(u_0) 
        \quad \mbox{ for every } \quad T>0.\]
        \item[(iii)] \[
        \frac{1}{t_2-t_1}\int_{t_1}^{t_2} \FenR(u(t))\,dt \leq \FenR(u_0) \quad 
    \mbox{ for every } \quad 0\leq t_1<t_2<\infty. 
        \] 
    \end{itemize}
\end{thm}
\subsection{\texorpdfstring{$\tvG$}{tvg}-Denoising Flow} We describe here another nonlocal version of the gradient flow to a BV based image restoration model, viz., the general ROF model. Consider the regression model: $$\mathfrak{R}(u):=\int_{\Omega}\mathbf{R}(x,u(x))\,dx,$$ where 
\begin{align}\label{convexdenoise}\begin{cases}
    \mathbf{R}:\Omega\times\bbR\to[0,\infty) \text{ is a Carath\'eodory integrand}\\
    \text{ such that for a.e. } x\in\Omega \text{ the partial map } \bbR\ni u\mapsto \mathbf{R}(x,u) \text{ is convex.}
\end{cases}
\end{align}
One defines a nonlocal version of the generalised (ROF) model by defining the functional
\begin{align}
    \FenG(u):=\tvG(u,\bbr)+\mathfrak{R}(u)=\tvG(u,\bbr)+\int_{\Omega}\mathbf{R}(x,u(x))\,dx,
\end{align} where $\tvG(\cdot)$ is a version of nonlocal BV space based on a Gagliardo-norm setting. The definition of this space will appear in the next section. In this article, we consider a gradient flow to this model. A particular difficulty arises in the specification of boundary data which is already a challenge in the local case. In the paper of \cite{Boegelein2015}, which we intend to generalise, the boundary data is prescribed on a thick boundary. This, of course, is familiar in the nonlocal settings where data is generally prescribed on complements of domains on which the PDE is posed. As a result, we will choose to work with nonlocal BV functions that are ``fixed-in-time" in the complement of the domain $\Omega$.

We will assume that the initial datum $u_0\in L^2(\Omega)\cap \bvG(\bbr)$ has a finite $\mathfrak{R}$-energy, i.e.,
\begin{align}\label{finiteRenergy}
    \mathfrak{R}(u_0)=\int_\Omega \Ren(x,u_0(x))\,dx<\infty,
\end{align} which leads to a finite $\FenG$-energy for $u_0$, i.e.,
\begin{align}\label{finiteRenergy2}
    \FenG(u_0)=\tvG(u_0,\bbr)+\int_\Omega \Ren(x,u_0(x))\,dx<\infty.
\end{align}
It is worthwhile to remember that $u_0$ is defined on $\bbr$ and $\Ren$ is defined on $\Omega\times\bbR$.

\begin{defn}\label{def:solg}
    Assume that the initial-boundary data $u_0\in L^2(\Omega)\cap \bvG(\bbr)$ satisfies \cref{finiteRenergy}. A function $u: \bbr\times[0,T] \rightarrow \mathbb{R}$ for $0<T<\infty$ is called a variational solution in $\Omega_T$ for the $\tvG$-denoising flow provided
    \[
    u \in \lwbvGz \cap C^0([0,T];L^2(\Omega))
    \]
    and the following variational inequality holds:
    \begin{gather}
    \label{eq:solg}
    \begin{aligned}
        \int_0^T \FenG(u(t)) \,dt \leq& \int_0^T\int_{\Omega}\partial_t v(v-u) \,dx \,dt + \int_0^T \FenG(v(t))\,dt\\
        &-\frac{1}{2}\|(v-u)(T)\|_{L^2(\Omega)}^2+\frac{1}{2}\|v(0)-u_0\|_{L^2(\Omega)}^2
    \end{aligned}
    \end{gather}
    for any $v\in \lwbvGz$ with $\partial_tv\in L^2(\Omega_T)$ and $v(0) \in L^2(\Omega)$. If $u$ is a variational solution in $\Omega_T$ for every $T>0$, we say that $u$ is a global variational solution. 
\end{defn}

The initial-boundary datum $u_0$ is a valid comparison map in the definition of variational solution as in \cref{def:solg} once we define its time independent extension $v(x,t)=u_0(x)$ for every $t\in [0,T]$. As a consequence, we obtain the following property of variational solutions. Due to the finite $\FenG$-energy of $u_0$, once we substitute the time-independent extension of $u_0$ as a comparison map in \cref{def:solg}, we obtain a finite energy bound for any solution $u$ when it exists:
\begin{align}
    \int_0^T\FenG(u(t))\,dt\leq \int_0^T\FenG(u_0)\,dt<\infty.
\end{align}

The following are the main results regarding the $\tvG$-denoising flow.

\begin{thm}[Existence]\label{existence2}
    Suppose that $\Ren(x,u)$ is a Carath\'eodory integrand as in \cref{convexdenoise} and $u_0$ satisfies \cref{finiteRenergy}. Then there exists a global variational solution $u$ to the $\tvG$-denoising flow in the sense of \cref{def:solg}.
\end{thm}
\begin{thm}[Uniqueness]\label{uniqueness2}
    Suppose that $\Ren(x,u)$ is a Carath\'eodory integrand as in \cref{convexdenoise}, and $u_0$ satisfies \cref{finiteRenergy}. Then there exists at most one global variational solution $u$ $\tvG$-denoising flow in the sense of \cref{def:solg}.
\end{thm}
\begin{thm}[Time Regularity]\label{timeregularity2}
    Suppose that $\Ren(x,u)$ is a Carath\'eodory integrand as in \cref{convexdenoise} and $u_0$ satisfies \cref{finiteRenergy}. Then any global variational solution $u$ $\tvG$-denoising flow as in \cref{def:solg} satisfies:
    \begin{enumerate}
        \item[(i)] $\partial_t u\in L^2(\Omega) \mbox{ and }u \in C^{0,\frac{1}{2}}([0,T];L^2(\Omega))\mbox{ for every } T>0.$
        \item[(ii)] $\int_0^T\int_{\Omega}|\partial_t u|^2 \,dx\,dS \leq \FenG(u_0) \mbox{ for every } T>0.$
        \item[(iii)] For every $t_1,t_2\in\bbR$ satisfying $0\leq t_1<t_2<\infty$, it holds that
        \begin{align}
            \frac{1}{t_2-t_1}\int_{t_1}^{t_2} \FenG(u(t))\,dt \leq \FenG(u_0).
        \end{align}
    \end{enumerate}
\end{thm}

\begin{rem}
    We have defined and proved properties of solutions to gradient flows for $\tvG$ functionals coupled with a denoising term and for $\tvr$ functionals coupled with a deblurring term. One can, of course, consider more varieties of functionals such as $\tvG$ functional with a beblurring term and $\tvr$ with denoising or other combinations of those. The techniques here would also apply to those models.
\end{rem}

\section{Notation and Preliminaries}
\subsection{Riesz BV}\label{sec:riesz}

We define the Riesz-BV spaces as follows. Given $\alpha\in (0,n)$, we let
\begin{align}
    I_{\alpha}f(x):=2^{-\alpha}\pi^{-n/2}\frac{\Gamma\left(\frac{n-\alpha}{2}\right)}{\Gamma\left(\frac{\alpha}{2}\right)}\int_{\bbr}\frac{f(y)}{|x-y|^{n-\alpha}}\,dy,\,x\in\bbr
\end{align} be the Riesz potential of order $\alpha$ of $f\in C^{\infty}_c(\bbr,\bbr)$. One way to define a fractional divergence is
\begin{align*}
    \divr f = \text{div}I_{1-\alpha}f.
\end{align*}
It may be seen that for any measurable function $F:\mathbb{R}^N\rightarrow \mathbb{R}^N$ and $\alpha \in (0,1]$ one has
\[
\nabla_R^{\alpha} \cdot F (x) = C(N,\alpha)\int_{\mathbb{R}^N}\frac{\left(F(x)-F(y)\right) \cdot \left(x-y\right)}{|x-y|^{N+\alpha+1}}.
\]
More details about the Riesz type fractional derivatives as well as motivation may be found in \cite{Shieh2014}. For $f \in L^1(\mathbb{R}^N)$ and $\alpha \in (0,1]$ we define the Riesz total variation of $f$ by
\[
TV_{R}^{\alpha}(f;\mathbb{R}^N) = \sup\left\{\int_{\mathbb{R}^N} f\nabla_R^{\alpha} \cdot \Psi \,dx : \Psi \in C_c^1(\mathbb{R}^N;\mathbb{R}^N), |\Psi|_{\infty} \leq 1\right\}.
\]
For $\Omega \subset \mathbb{R}^N$ and $f \in L^1(\Omega)$ we define
\[
TV_{R}^{\alpha}(f;\Omega) = TV_{R}^{\alpha}(f\mathbbm{1}_{\Omega};\mathbb{R}^N)
\]
where $\mathbbm{1}_{\Omega}$ is the indicator function of the set $\Omega$ and $f\mathbbm{1}_{\Omega}$ is the extension by zero of $f$ to $\mathbb{R}^N$. For $\Omega \subset \mathbb{R}^N$ we define the Riesz-BV spaces by
\[
\bvr(\Omega) = \left\{f \in L^1(\mathbb{R}^N): f = 0 \mbox{ in } \mathbb{R}^N\setminus\Omega, TV_{R}^{\alpha}(f;\Omega) < +\infty\right\}.
\]
with the norm:
\[
\|f\|_{\bvr(\Omega)} = \|f\|_{L^1} + TV_{R}^{\alpha}(f;\Omega).
\]
For ease of notation, we write $TV_{R}^{\alpha}(f;\Omega)$ as $TV_{R}^{\alpha}(f)$. Finally, as noted in \cite[Corollary 2.7]{Antil2024}, $\bvr(\Omega)$ is a Banach space.

\subsection{Gagliardo BV spaces}\label{sec:gag} In order to define the Gagliardo type fractional BV spaces, one needs the Gagliardo type fractional derivative $\textbf{d}_\alpha$ which takes an $\mathcal{L}^N$-measurable function $f:\mathbb{R}^N\to\mathbb{R}$ into a non-local vector field \[(\textbf{d}_\alpha f)(x,y):=\frac{f(x)-f(y)}{|x-y|^\alpha}.\] Here $\mathcal{L}^N$ refers to the $N$-dimensional Lebesgue measure. Arguing through duality (as in \cite{Antil2024}), one arrives at the following definition for the Gagliardo divergence of a nonlocal vector field. Let $F:\bbr\times\bbr\to\bbR$ be a $\mathcal{L}^N\times\mathcal{L}^N$ measurable map, then its Gagliargo divergence is defined to be
\[\divG F(x):=-\int_\bbr \frac{F(x,y)-F(y,x)}{|x-y|^{N+\alpha}}\,dy.\]

\begin{defn}
    Let $f\in \lone_{\textbf{loc}}(\bbr)$. We define the Gagliargo total variation of $f$ in $\bbr$ as 
    \begin{align*}
        \tvG(f;\bbr) = \sup\left\{ \int_\bbr f(x) (\divG \Phi)(x)\,dx: \Phi\in C_c^\infty(\bbr\times\bbr), \|\Phi\|_{L^\infty(\bbr\times\bbr)} \leq 1 \right\}
    \end{align*}
    We say that $f\in \bvG(f;\bbr)$ if $\|f\|_{\bvG(\bbr)}: = \|f\|_{\lone(\bbr)}+\tvG(f;\bbr)$ is finite. 

    In the future, we will be making use of functions $u$ that are `fixed' outside of a bounded open set $\Omega$. Let $\Omega$ be a bounded open set in $\bbr$. Let $u_0\in \bvG(\bbr)$. We define
    \begin{align}
        \bvG(\Omega;u_0):= \left\{ u\in\bvG(\bbr): u(x)=u_0(x) \text{ a.e. } x\in\bbr\setminus\Omega\right\}.
    \end{align}
    
    We also say that $f\in \textbf{W}^{\alpha,1}(\bbr)$ if $f\in \lone(\bbr)$ and 
    \[\displaystyle [f]_{\textbf{W}^{\alpha,1}(\bbr)}:=\iint_{\bbr\times\bbr} \frac{|f(x)-f(y)|}{|x-y|^{n+\alpha}}\,dy\,dx<\infty.\]
    It is proved in \cite[Theorem 3.4]{Antil2024} that $\bvG(\bbr)=W^{\alpha,1}(\bbr)$ with equality of the suggested norms. In particular, $\bvG(\bbr)$ is a Banach space.
\end{defn}

\subsection{Evolutionary fractional BV spaces} 
\begin{defn} 
Let $\Omega$ be an open subset of $\bbr$. The space $\lwbv$ consists of functions $v: [0,T] \rightarrow \bvr(\Omega)$ such that
\begin{itemize}
    \item[(A)]  We have $\displaystyle v \in \lone(\Omega\times (0,T)),$
    \item[(B)] For every $\Psi \in C_c^1(\bbr,\bbr)$ the maps:
    \[
    [0,T] \ni t \mapsto \int_{\Omega} v(t) \nabla_R^{\alpha} \cdot \Psi\,dx
    \]
    are measurable and
    \item[(C)] We have $\int_0^T \tvr(v(t)) \,dt <+\infty.$ 
\end{itemize}
\end{defn}

Similarly, we define
\begin{defn}
    Let $\Omega$ be an open subset of $\bbr$. The space $\lwbvGz$
consists of functions $v: [0,T] \rightarrow \bvG(\Omega;u_0)$ such that
\begin{itemize}
    \item[(A)]  $v \in \textbf{L}^{\textbf{1}}(\bbr\times (0,T))$,
    \item[(B)] For every $\Phi \in C_c^1(\bbr\times\bbr;\bbR)$ the maps:
    $\displaystyle[0,T] \ni t \mapsto \int_{\bbr} v(t,x)\, (\divG \Phi)(x)\,dx$
    are measurable and
    \item[(C)] The quantity $\displaystyle\int_0^T \tvG(v(t);\bbr) \,dt$ is finite.
\end{itemize}
\end{defn}

We have the following closure properties for the evolutionary fractional BV spaces.

\begin{lem}[Closure in $\lwbv$]\label{lem:closure1}
    Let $\Omega$ be an open subset of $\bbr$. Suppose that $\displaystyle u_n \in \lwbv$ for $n \geq 1$ satisfies:
    \begin{itemize}
        \item[(i)] $\sup_{n \geq 1} \int_0^T \tvr(u_n(t))\,dt < + \infty$ and
        \item[(ii)] $u_n \rightarrow u$ in $L^1(\Omega_T).$
        Then \[u \in \lwbv.\]
    \end{itemize}
\end{lem}
\begin{proof}
    Through a standard argument involving Fubini's theorem, there is a subsequence along which $u_n(t)\to u(t)$ as $n\to\infty$ in $L^1(\Omega)$ a.e. $t\in [0,T]$. Therefore, for any $\Psi\in C^1_c(\bbr,\bbr)$, it holds for a.e. $t\in[0,T]$,
    \begin{align*}
        \lim_{n\to\infty}\int_\Omega u_n(x,t) \divr\cdot\Psi(x)\,dx = \int_\Omega u(x,t) \divr\cdot\Psi(x)\,dx.
    \end{align*} Since the map $t\mapsto \int_\Omega u_n(x,t) \divr\cdot\Psi(x)\,dx$ is measurable for $t\in [0,T]$, the limit $t\mapsto \int_\Omega u(x,t) \divr\cdot\Psi(x)\,dx$ is also measurable. For each $t\in [0,T]$, we apply lower semicontinuity of fractional $\bvr$ norm with respect to $L^1$ convergence (\cite[Proposition 2.6]{Antil2024}) followed by Fatou's lemma to get
    \begin{align*}
        \int_0^T \tvr(u)\,dt \leq \int_0^T \liminf_{n\to\infty} \,\tvr(u_n)\,dt \leq \liminf_{n\to\infty} \,\int_0^T \tvr(u_n)\,dt
    \end{align*} which is finite. This implies, in particular, that $\tvr(u(t))$ is finite a.e. $t\in[0,T]$ whence $u(t)\in \bvr(\Omega)$. Altogether this proves that $u\in \lwbv$.
\end{proof}
\begin{lem}[Closure in $\lwbvGz$]\label{lem:closure2}
    Let $\Omega$ be an open subset of $\bbr$. Suppose that $\displaystyle u_n \in \lwbvGz$ for $n \geq 1$ satisfies:
    \begin{itemize}
        \item[(i)] $\sup_{n \geq 1} \int_0^T \tvG(u_n(t);\bbr)\,dt < + \infty$ and
        \item[(ii)] $u_n \rightarrow u$ in $L^1(\bbr\times(0,T)).$
        Then \[u \in \lwbvGz.\]
    \end{itemize}
\end{lem}
\begin{proof}
    The proof is identical to the proof of \cref{lem:closure1} except, we use \cite[Lemma 3.8]{Antil2024} for lower semicontinuity of BV seminorm in $\bvG(\bbr)$ with respect to $L^1$-norm convergence.
\end{proof}
\subsection{Time mollification}
A difficulty in working with parabolic problems is the absence of time-regularity which prevents the use of the solutions themselves as test functions in obtaining energy inequalities. To overcome this deficit, it is standard to use time-mollification such as the one described below. Let $V$ be a Banach space with $v_0 \in V$. For $v \in L^r(0,T;V)$ with $r \in [1,+\infty]$ we define the time mollification of $v$ as follows: for a small $0<h\leq T$ and any $t \in [0,T]$ 
\begin{align}\label{timemolli}
[v]_h(t) = \exp(-t/h)v_0 + \frac{1}{h}\int_0^t \exp\left(\tfrac{s-t}{h}\right)v(s) \,ds.
\end{align}
We have the following basic properties of the mollification whose proofs may be found in \cite[Appendix B]{Bogelein2013}.

\begin{thm}\label{timemollithm1}
    Suppose $V$ is a separable Banach space and $v_0\in V$. If $v\in L^r(0,T;V)$ for some $r\geq 1$, then the mollification $[v]_h$ defined by \cref{timemolli} satisfies $[v]_h\in C^\infty([0,T];V)$ and for any $t_0\in (0,T]$ there holds for $1\leq r<\infty$
    \begin{align*}
        \|[v]_h\|_{L^r(0,t_0;V)}\leq\|v\|_{L^r(0,t_0;V)}+\left[ \frac{h}{r}\left(1-\exp\left(-\frac{t_0r}{h}\right)\right) \right]^{\frac1r} \|v_0\|_V,
    \end{align*} and for $r=\infty$, we have 
    \begin{align*}
        \|[v]_h\|_{L^\infty(0,t_0;V)}\leq\|v\|_{L^\infty(0,t_0;V)}+ \|v_0\|_V.
    \end{align*} Moreover, we have $[v]_h(0)\to v$ in $L^r(0,T;V)$ as $h\to 0$. If $v\in C^0([0,T];V)$, then $[v]_h\in C^0([0,T];V)$, $[v]_h(0)=v_0$ and $[v]_h\to v$ in $C^0([0,T];V)$ as $h\to0$.
\end{thm}

\begin{thm}\label{timemollithem2}
    Let $V$ be a separable Banach space and $r\geq 1$ Assume that $v\in L^r(0,T;V)$ with $\de_t v\in L^r(0,T;V)$. Then the function $[v]_h$ which is the time-mollification of $v$ as defined in \cref{timemolli} satisfies
    \begin{align*}
        \de_t[v]_h(t)=\frac1h\int_0^texp\left(\frac{s-t}{h}\right)\de_sv(s)\,ds,
    \end{align*}  the identity
    \begin{align}\label{ODEmolli}
        \de_t[u]_h=-\frac1h([u]_h-u).
    \end{align}
    with the bound
    \begin{align*}
        \|\de_t[v]_h\|_{L^r(0,T;V)}\leq \|v\|_{L^r(0,T;V)}.
    \end{align*}
\end{thm}

Below, we state two results about the convergence of the denoising and deblurring functionals of time-mollified solutions. These will be required in the passage to limit. The proof of \cref{convdenoiselem} can be found in \cite[Lemma 2.3]{Bogelein2014} with some obvious modifications whereas the proof of \cref{convdeblurlem} can be found in \cite[Lemma 2.5]{Boegelein2015}.

\begin{lem}[Convergence of denoising functionals]\label{convdenoiselem}
    Let $T>0$. Assume that $S:\Omega\times\bbR\to\bbR$ is a Carath\'eodory integrand as in \cref{convexdenoise} and that $v\in L^1(\Omega_T)$ with $S(\cdot,v)\in L^1(\Omega_T)$, $v_0\in L^1(\Omega)$ and $S(\cdot,v_0)\in L^1(\Omega)$. Then $S(\cdot,[v]_h)\in L^1(\Omega_T)$. Finally,
    \begin{align*}
        \lim_{h\to0}\int_0^T\int_\Omega S(x,[v]_h(x,t))\,dx\,dt = \int_0^T\int_\Omega S(x,v(x,t))\,dx\,dt
    \end{align*}
\end{lem}

\begin{lem}[Convergence of deblurring functionals]\label{convdeblurlem}
    Let $T>0$ and let $K$ be as in \cref{deblurfunc}. Furthermore let $v_0\in L^1(\Omega)$, $v \in L^1(\Omega_T)$, and $u_0\in L^2(\Omega)$ and define $[v]_h$ as in \cref{timemolli}. Then $K[[v]_h(t)]-u_0]\in L^2(\Omega_T)$, and
    \begin{align*}
        |K[[v]_h(t)]-u_0|^2\leq[|K[v(\cdot)]-u_0|^2]_h(t),
    \end{align*} where $[|K[v](\cdot)-u_0|^2]_h(t)$ is defined via \cref{timemolli}
 with $|K[v_0]-u_0|^2$ in place of $v_0$. Finally, we have
    \begin{align*}
        \lim_{h\to0}\int_0^T\int_\Omega |K[[v]_h(t)]-u_0|^2\,dx\,dt = \int_0^T\int_\Omega |K[v(t)]-u_0|^2\,dx\,dt
    \end{align*}
\end{lem}

Similarly, the \textbf{TV} seminorms of time-mollified solutions also converge as the mollification parameter $h$ goes to $0$. These results will be required later.

\begin{lem}\label{convtvrfunc}
    Let $T>0$, $v_0 \in \bvr(\Omega)$ and $v \in L_w^1\left(0,T;\bvr(\Omega)\right)$. Then 
    \begin{itemize}
        \item[(i)] \[[v]_h \in L_w^1\left(0,T;\bvr(\Omega)\right),\]
        \item[(ii)] \[
        \tvr([[v]_h(t)) \leq [\tvr(v)]_h(t) \quad \mbox{ for every } \quad 0<t<T, \quad \mbox{ and }
        \]
        \item[(iii)] \[
        \lim_{h \searrow 0}\int_0^T \tvr([v]_h(t))\,dt = \int_0^T \tvr(v)(t) \,dt. 
        \]
    \end{itemize}
\end{lem}

\begin{proof}
    An application of \cref{timemollithm1} ensures that $[v]_h\to v$ in $L^1(0,T;L^1(\Omega))$ as $h\to0$. As a result, along a subsequence $[v]_h(t)\to v(t)$ for almost every $t\in[0,T]$. By definition, for $\Psi\in C^1_c(\bbr,\bbr)$, we have
    \begin{align*}
        \int_\Omega [v]_h(x,t)\divr\cdot \Psi(x)\,dx= \exp\left(-\frac{t}{h}\right) \int_\Omega v_0 \divr\cdot \Psi\,dx + \frac1h\int_0^t \exp\left(\frac{s-t}{h}\right)\int_\Omega v(s)\divr\cdot\Psi\,dx\,ds
    \end{align*} through an application of Fubini's theorem. In particular, this implies that the map
    \begin{align*}
        (0,T)\ni t\mapsto \int_\Omega [v]_h(x,t)\divr\cdot \Psi(x)\,dx
    \end{align*} is measurable. Also, once we take the supremum over all $\Psi\in C^1_c(\bbr,\bbr)$ with $\|\Psi\|_\infty\leq 1$, we obtain the bound
    \begin{align*}
        \tvr([v]_h(t))\leq \exp\left(-\frac{t}{h}\right) \tvr(v_0) + \frac1h\int_0^t \exp\left(\frac{s-t}{h}\right)\tvr(v(s))\,ds=[\tvr(v(t)]_h\leq\infty
    \end{align*} which proves that $[v]_h(t)\in \bvr(\Omega)$ for almost every $t$. 
    Now applying \cref{timemollithm1} to the function $t\mapsto \tvr(v)$ (this map belongs to $L^1(0,T;\bbR)$ since $v\in \lwbv$), we obtain the bound
    \begin{align}\label{bvboundvh}
        \int_0^T \tvr([v]_h(t))\,dt \leq h\tvr(v_0) + \int_0^T \tvr(v(t))\,dt.
    \end{align} This proves that $[v]_h\in \lwbv$. Finally, by the lower semicontinuity of total variation $\tvr$ as in \cite[Proposition 2.6]{Antil2024}, Fatou's lemma and \cref{bvboundvh}, we get
    \begin{align*}
        \int_0^T \tvr(v(t))\,dt \leq \int_0^T \liminf_{h\to0} \tvr([v]_h(t))\,dt \leq \liminf_{h\to 0} \int_0^T \tvr([v]_h(t))\,dt \stackrel{\cref{bvboundvh}}{\leq} \int_0^T \tvr(v(t))\,dt.
    \end{align*} This completes the proof of the lemma.
\end{proof}

\begin{lem}\label{convtvGfunc}
    Let $T>0$, $v_0 \in \bvG(\bbr)$ and $v \in \lwbvGzvz$. Then 
    \begin{itemize}
        \item[(i)] \[[v]_h \in \lwbvGzvz,\]
        \item[(ii)] \[
        \tvG([[v]_h(t)) \leq [\tvG(v)]_h(t) \quad \mbox{ for every } \quad 0<t<T, \quad \mbox{ and }
        \]
        \item[(iii)] \[
        \lim_{h \searrow 0}\int_0^T \tvG([v]_h(t))\,dt = \int_0^T \tvG(v)(t) \,dt. 
                \]
   \end{itemize}
\end{lem}

\begin{proof}
 The proof is identical with the proof of \cref{convtvrfunc} except for the use of \cite[Lemma 3.8]{Antil2024} for the lower semicontinuity of $\tvG$ seminorm with respect to $L^1$ convergence. We additionally need to check that for a.e. $t\in[0,T]$ and $x\in\bbr\setminus\Omega$, it holds that $[v]_h(x,t)=v_0(x)$. This follows from the definition \cref{timemolli} whence for $x\notin\Omega$, we have
 \begin{align*}
     [v]_h(x,t) = \exp(-t/h)v_0(x) + \frac{1}{h}\int_0^t \exp\left(\tfrac{s-t}{h}\right)v_0(x) \,ds=v_0(x),
 \end{align*} where we use the fact that $v(x,t)=v_0(x)$ a.e. $x\in\bbr\setminus\Omega$ and $t\in[0,T]$.
\end{proof}

The following two corollaries are immediate consequences of \cref{convdeblurlem}, \cref{convtvrfunc} and \cref{convdenoiselem}, \cref{convtvGfunc}.

\begin{cor}\label{convfuncR}
    Suppose that $v_0 \in \bvr(\Omega)$ with $F(v_0) < +\infty$. Let $v \in L_w^1\left(0,T;\bvr(\Omega)\right)$ with $\int_0^T F(v(t)) \,dt < +\infty.$ Then 
    \begin{itemize}
        \item[(i)] \[
        [v]_h \in L_w^1\left(0,T;\bvr(\Omega)\right),
        \]
        \item[(ii)] \[
         \FenR([v]_h(t)) \leq [\FenR(v)]_h(t) \quad \mbox{ for every } \quad 0<t<T, \quad \mbox{ and }
        \] 
        \item[(iii)] \[
        \lim_{h \searrow 0}\int_0^T \FenR([v]_h(t))\,dt = \int_0^T \FenR(v)(t) \,dt. 
        \]
    \end{itemize}
\end{cor}

\begin{cor}\label{convfuncG}
    Suppose that $v_0 \in \bvG(\bbR)$ with $\FenG(v_0) < +\infty$. Let $v \in \lwbvGzvz$ with $\int_0^T \FenG(v(t)) \,dt < +\infty.$ Then 
    \begin{itemize}
        \item[(i)] \[
        [v]_h \in \lwbvGzvz,
        \]
        \item[(ii)] \[
         \FenG([v]_h(t)) \leq [\FenG(v)]_h(t) \quad \mbox{ for every } \quad 0<t<T, \quad \mbox{ and }
        \] 
        \item[(iii)] \[
        \lim_{h \searrow 0}\int_0^T \FenG([v]_h(t))\,dt = \int_0^T \FenG(v)(t) \,dt. 
        \]
    \end{itemize}
\end{cor}

\subsection{Localization to subcylinders}\label{localizesub} In this section, we will prove that if $u\in \lwbvGz$ is a variational solution in $\Omega_T$ in the sense of \cref{def:solg} then for any $s_1,s_2\in\bbR$ satisfying $0\leq s_1<s_2\leq T$, the function $u$ restricted to the interval $(s_1,s_2)$ is a variational solution in $\Omega\times (s_1,s_2)$. 

To this end, we take a function $v\in \textbf{L}^1_{\textbf{w}}(s_1,s_2;\bvG(\Omega;u_0))$ with $\partial_t v\in L^2(\Omega\times (s_1,s_2))$, $v(s_1)\in L^2(\Omega)$ and a cutoff function defined for a fixed $\theta\in(0,\frac12(s_2-s_1))$
\begin{align*}
    \zeta_\theta(t):=\begin{cases}
        0 &\mbox{ for }t\in[0,s_1],\\
        \frac1\theta(t-t_1)&\mbox{ for }t\in (s_1,s_1+\theta),\\
        1&\mbox{ for }t\in[s_1+\theta,s_2-\theta],\\
        \frac1\theta(s_2-t)&\mbox{ for }t\in(s_2-\theta,s_2),\\
        0&\mbox{ for }t\in[s_2,T].
    \end{cases}
\end{align*}
We will test the definition of variational solutions as in \cref{def:solg} with the comparison map defined by $\tilde{v}:=\zeta_\theta v + (1-\zeta_\theta)[u]_h$. The function $\zeta_\theta$ has been extended outside of $(s_1,s_2)$ by $0$.

We begin by checking that $\tilde{v}$ is a valid comparison map for \cref{def:solg}. From \cref{convtvGfunc}, we have $[u]_h\in\lwbvGz$, therefore, $\tilde{v}\in\lwbvGz$. The boundary condition is inherited since $\tilde{v}$ is a convex combination of two functions both satisfying the same boundary condition. Also, by \cref{timemollithem2}, $\de_t \tilde{v}\in L^2(\Omega_T)$ since $\de_t[u]_h\in L^2(\Omega_T)$. We finally need to show that $\tilde{v}(0)\in L^2(\Omega)$. This is true since $[u]_h\in C^0([0,T];\Omega)$ on account of $u\in C^0([0,T];\Omega)$. We will now make use of $\tilde{v}$ in \cref{eq:solg}.
\begin{gather}\label{timelocali1}
    \begin{aligned}
        \int_0^T \FenG(u(t)) \,dt \leq& \int_0^T\int_{\Omega}\partial_t \tilde{v}(\tilde{v}-u) \,dx \,dt + \int_0^T \FenG(\tilde{v}(t))\,dt\\
        &-\frac{1}{2}\|(\tilde{v}-u)(T)\|_{L^2(\Omega)}^2+\frac{1}{2}\|\tilde{v}(0)-u_0\|_{L^2(\Omega)}^2.
    \end{aligned}
\end{gather}

We first note that the first integrand on the right hand side (RHS) can be written as
\begin{align*}
    \de_t\tilde{v}(\tilde{v}-u) &= \de_t(\zeta_\theta v+(1-\zeta_\theta)[u]_h)(\zeta_\theta v+(1-\zeta_\theta)[u]_h-u)\\
    &=\zeta_\theta \de_t v(\zeta_\theta v+(1-\zeta_\theta)[u]_h-u+\zeta_\theta u-\zeta_\theta u)\\
    &\qquad (1-\zeta_\theta)\de_t[u]_h(([u]_h-u)+\zeta_\theta(v-[u]_h))\\
    &\qquad\qquad \de_t\zeta_\theta (v-[u]_h)(([u]_h-u)+\zeta_\theta(v-[u]_h))\\
    &=\zeta_\theta^2\de_t v(v-u) +\zeta_\theta(1-\zeta_\theta)\de_t v([u]_h-u)+(1-\zeta_\theta)(\de_t [u]_h)([u]_h-u)\\
    &\qquad +\zeta_\theta (1-\zeta_\theta)(\de_t[u]_h)(v-[u]_h) +\zeta_{\theta}^{'}(v-[u]_h)([u]_h-u) + \zeta_{\theta}^{'}\zeta_\theta(v-[u]_h)^2\\
    &=\zeta_\theta^2\de_t v(v-u) +(1-\zeta_\theta)\de_t [u]_h([u]_h-u)\\
    &\qquad+ \zeta_\theta (1-\zeta_\theta)(\de_t v([u]_h-u)+\de_t[u]_h(v-[u]_h))\\
    &\qquad\qquad \zeta_{\theta}^{'}(v-[u]_h)([u]_h-u) + \zeta_{\theta}^{'}\zeta_\theta(v-[u]_h)^2.
\end{align*}
Therefore, we write the first term on the RHS of \cref{timelocali1} as 
\begin{align*}
    \int_0^T\int_{\Omega}\partial_t \tilde{v}(\tilde{v}-u) \,dx \,dt&=\underbrace{\int_{s_1}^{s_2}\int_{\Omega}\zeta_\theta^2\de_t v(v-u)\,dx \,dt}_{I_\theta}+\underbrace{\int_0^T\int_{\Omega}(1-\zeta_\theta)\de_t [u]_h([u]_h-u)\,dx \,dt}_{II_{\theta}}\\
    &\qquad+\underbrace{\int_{s_1}^{s_2}\int_{\Omega}\zeta_\theta (1-\zeta_\theta)(\de_t v([u]_h-u)+\de_t[u]_h(v-[u]_h))\,dx \,dt}_{III_{\theta}}\\
    &\qquad\qquad+\underbrace{\int_{s_1}^{s_2}\int_{\Omega}\zeta_{\theta}^{'}(v-[u]_h)([u]_h-u)\,dx \,dt}_{IV_\theta}+\underbrace{\int_{s_1}^{s_2}\int_{\Omega}\zeta_{\theta}^{'}\zeta_\theta(v-[u]_h)^2\,dx \,dt}_{V_\theta}.
\end{align*}
Through applications of Dominated Convergence Theorem, we have
\begin{align*}
    \lim_{\theta\downarrow 0} I_\theta &=\int_{s_1}^{s_2}\int_{\Omega}\de_t v(v-u)\,dx \,dt,\\
    \lim_{\theta\downarrow 0} II_\theta &=\int_{[0,s_1]\cup[s_2,T]}\int_{\Omega}\de_t [u]_h([u]_h-u)\,dx \,dt,\\
    \lim_{\theta\downarrow 0} III_\theta &=0,\\
    \lim_{\theta\downarrow 0} IV_\theta &=\int_{\Omega\times{\{s_1\}}}(v-[u]_h)([u]_h-u)\,dx-\int_{\Omega\times{\{s_2\}}}(v-[u]_h)([u]_h-u)\,dx,\\
    \lim_{\theta\downarrow 0} V_\theta &=\lim_{\theta\downarrow 0}\int_{s_1}^{s_1+\theta}\frac12\de_t(\zeta_{\theta}^2)\int_{\Omega}(v-[u]_h)^2\,dx \,dt+\lim_{\theta\downarrow 0}\int_{s_2-\theta}^{s_2}\frac12\de_t(\zeta_{\theta}^2)\int_{\Omega}(v-[u]_h)^2\,dx \,dt\\
    &=\frac12\int_{\Omega}(v-[u]_h)^2(x,s_1)\,dx-\frac12\int_{\Omega}(v-[u]_h)^2(x,s_2)\,dx,
\end{align*} the last two additionally using Lebesgue differentiation theorem in time for continuous functions.

At the same time, by convexity of $\FenG$, we get 
\begin{align}
    \FenG(\tilde{v}(t))\leq \zeta_\theta(t)\FenG(v(t))+(1-\zeta_\theta(t))\FenG([u]_h(t))
\end{align}
Finally, the last two terms on the RHS of \cref{timelocali1} become
\begin{align*}
    -\frac{1}{2}\|([u]_h-u)(T)\|_{L^2(\Omega)}^2+\frac{1}{2}\|[u]_h(0)-u_0\|_{L^2(\Omega)}^2
\end{align*} due to $\zeta_\theta(T)=\zeta_\theta(0)=0$.

Substituting in \cref{timelocali1} and passing to the limit as $\theta\to 0$, we obtain

\begin{gather}\label{timelocali2}
    \begin{aligned}
        \int_{s_1}^{s_2} \FenG(u(t)) \,dt \leq& \int_{s_1}^{s_2}\int_{\Omega}\partial_t v(v-u) \,dx \,dt + \int_{s_1}^{s_2} \FenG(v(t))\,dt\\
        &+\int_{[0,s_1]\cup[s_2,T]}\int_{\Omega}\de_t [u]_h([u]_h-u)\,dx+\left(\FenG([u]_h(t))-\FenG(u(t)) \right)\,dt\\
        &+\int_{\Omega\times{\{s_1\}}}(v-[u]_h)([u]_h-u)\,dx-\int_{\Omega\times{\{s_2\}}}(v-[u]_h)([u]_h-u)\,dx\\
        &+\frac12\int_{\Omega}(v-[u]_h)^2(x,s_1)\,dx-\frac12\int_{\Omega}(v-[u]_h)^2(x,s_2)\,dx\\
        &-\frac{1}{2}\|([u]_h-u)(T)\|_{L^2(\Omega)}^2+\frac{1}{2}\|[u]_h(0)-u_0\|_{L^2(\Omega)}^2.
    \end{aligned}
\end{gather}
Finally, we can pass to the limit in \cref{timelocali2} as $h\to0$ by applying \cref{timemollithm1} and \cref{convfuncG}. We also drop the term $\int_{[0,s_1]\cup[s_2,T]}\int_{\Omega}\de_t [u]_h([u]_h-u)\,dx$ since it is non-positive due to \cref{ODEmolli}. We are left with
\begin{gather}\label{timelocali3}
    \begin{aligned}
        \int_{s_1}^{s_2} \FenG(u(t)) \,dt \leq& \int_{s_1}^{s_2}\int_{\Omega}\partial_t v(v-u) \,dx \,dt + \int_{s_1}^{s_2} \FenG(v(t))\,dt\\
        &+\frac12\int_{\Omega}(v-u)^2(x,s_1)\,dx-\frac12\int_{\Omega}(v-u)^2(x,s_2)\,dx,    \end{aligned}
\end{gather} which is the definition of variational solution on the time interval $(s_1,s_2)$ as was required.

\begin{rem}\label{localiforsolr}
    The localization described in this section also works for solutions $u\in\lwbv$ in the sense of \cref{def:solr} with the obvious modifications.
\end{rem}

\subsection{Initial condition}
In this section, we will prove that the variational solutions satisfy the initial condition in the $L^2$ sense. More precisely
\begin{lem}
    Let $u$ be a variational solution to the $\tvr$-deblurring flow in the sense of \cref{def:solr}. Then the initial condition is achieved in the $L^2$-sense, which is to say that 
    \begin{align*}
            \lim_{t \searrow 0}\|u(t)-u_0\|_{L^2(\Omega)}^2 = 0.
    \end{align*}
\end{lem}
\begin{proof}
    In \cref{localizesub}, we proved that $u$ is also a solution on subcylinders $\Omega\times[0,\tau]$ for any $\tau\in (0,T)$. We test the variational inequality \cref{timelocali3} with the time-independent extension of the initial-boundary data $u_0$ with $s_1=0$ and $s_2=\tau$, whence
    \begin{align}
        \int_{0}^{\tau} \FenR(u(t)) \,dt+\frac12\int_{\Omega}(u-u_0)^2(x,\tau)\,dx \leq& \int_{0}^{\tau} \FenR(u_0)\,dt=\tau\FenR(u_0)<\infty.
    \end{align} We obtain the desired convergence by taking limits on both sides as $\tau\to0$ and dropping the first term on the left hand side (LHS).
\end{proof}

\begin{rem}
    The variational solutions to $\tvG$ denoising flow in the sense of \cref{def:solg} also satisfies initial condition in the $L^2$-sense and the proof is similar.
\end{rem}

\section{Regularity in time} Regularity in time for solutions of parabolic equations is generally not available and hence we do not assume any such regularity for the variational solutions either (see \cref{def:solr} and \cref{def:solg}). However, on account of the time-independent boundary data, we are able to gain regularity in time as reported in \cref{timeregularity1} and \cref{timeregularity2}. Their proofs is the content of this section. The proofs are similar to those of \cite[Theorem 1.3]{Boegelein2015}. We repeat them here for the reader's convenience.

\begin{proof}(Proof of \cref{timeregularity1} and \cref{timeregularity2}) 
    Let $t_1,t_2\in\bbR$ with $0\leq t_1<t_2\leq T$. The variational solution $u$ in the sense of \cref{def:solr} is also a solution on the subcylinder $\Omega\times(t_1,t_2)$. Define $\tilde{u}(s)=u(s+t_1)$ for $s\in(0,t_2-t_1)$. If $u(t_1)\in\bvr(\Omega)$, then $\tilde{u}$ satisfies the variational inequality \cref{eq:sol} in the cylinder $\Omega\times[t_1,t_2]$ with initial data $u(t_1)$ instead of $u_0$. Indeed, we choose such a $t_1$, which is possible to achieve since $u(t)\in\bvr(\Omega)$ for a.e. $t\in[0,T]$. Further, the function $v=[\tilde{u}]_h$ is an admissible comparison map since $v(0)=[\tilde{u}]_h(0)=u(t_1)$ as argued before in \cref{localizesub}. Testing with $v=[\tilde{u}]_h$, we drop the negative term from the RHS of the variational inequality and notice that the initial time term is zero on account of the initial data for $u$ and $v$ matching up as defined
    \begin{align*}
        -\int_0^{t_2-t_1}\int_\Omega \de_t[\tilde{u}]_h\left([\tilde{u}]_h-u\right)\,dx\,dt&\leq \int_0^{t_2-t_1}\FenR([\tilde{u}]_h(t))-\FenR(\tilde{u}(t))\,dt\\
        &\stackrel{\cref{convtvrfunc}}{\leq} \int_0^{t_2-t_1}[\FenR(\tilde{u}(t)]_h-\FenR(\tilde{u}(t))\,dt\\
        &\stackrel{\cref{ODEmolli}}{\leq} -h\int_0^{t_2-t_1}\de_t[\FenR(\tilde{u}(t))]_h\,dt\\
        &=h\left(\FenR(\tilde{u}(0))-[\FenR(\tilde{u}(t))]_h(t_2-t_1)\right).
    \end{align*}
    Using \cref{ODEmolli}, we get 
    \begin{align*}
        \int_0^{t_2-t_1}\int_\Omega |\de_t[\tilde{u}]_h|^2\,dx\,dt\leq \left(\FenR(\tilde{u}(0))-[\FenR(\tilde{u}(t))]_h(t_2-t_1)\right)\leq \FenR(\tilde{u}(0)).
    \end{align*}
    The previous display ensures the existence of $\de_t\tilde{u}\in L^2(\Omega\times(0,t_2-t_1)$ due to the uniformity of the estimate with respect to the parameter $h$. It comes with the quantitative estimate
    \begin{align*}
        \int_0^{t_2-t_1}\int_\Omega |\de_t\tilde{u}|^2\,dx\,dt\leq \FenR({u}(t_1)).
    \end{align*}
    Applying \cref{convtvrfunc} to the second last display, we obtain
    \begin{align*}
        \int_0^{t_2-t_1}\int_\Omega |\de_t\tilde{u}|^2\,dx\,dt\leq \FenR(\tilde{u}(0))-\FenR(\tilde{u}(t_2-t_1)).
    \end{align*} Reverting to $u$, we obtain
    \begin{align}\label{l2boundfordetu}
        \int_{t_1}^{t_2}\int_\Omega |\de_t{u}|^2\,dx\,dt\leq \FenR({u}(t_1))-\FenR({u}(t_2)).
    \end{align} The previous inequality holds a.e. $t_1, t_2\in [0,T]$. This proves part of the assertions in the statement of the theorem. 

    Next, let us estimate using H\"older's inequality as follows
    \begin{align*}
        \|u(t_2)-u(t_1)\|^2_{L^2(\Omega)}&=\int_{\Omega}\left|\int_{t_1}^{t_2}\de_t u\,dt\right|^2\,dx\\
        &\leq |t_2-t_1|\int_{t_1}^{t_2}\int_\Omega|\de_t u|^2\,dx\,dt\stackrel{\cref{l2boundfordetu}}{\leq} |t_2-t_1|\FenR(u(t_1)).
    \end{align*}
    Choosing $t_1=0$, we get
    \begin{align*}
        \int_\Omega|u(t)|^2\,dx\leq 2\int_\Omega|u_0|^2\,dx + 2\int_\Omega |u(t)-u(0)|^2\,dx\leq 2\|u_0\|^2_{L^2(\Omega)}+2t\FenR(u_0),
    \end{align*} which proves that $u\in C^{0,\frac12}([0,\tau],L^2(\Omega))$ for any $\tau\in (0,T]$.

    To prove the remaining estimate in \cref{timeregularity1}, with the time-regularity $\de_tu\in L^2(\Omega_T)$ in place, we can integrate by parts in \cref{eq:sol} to get
    \begin{align}\label{varineq2}
        \int_{0}^{\tau} \FenR(u(t)) \,dt \leq& \int_{0}^{\tau}\int_{\Omega}\partial_t u(v-u) \,dx \,dt + \int_{0}^{\tau} \FenR(v(t))\,dt
    \end{align} for any $\tau\in (0,T]$. Now, for $t_1,t_2\in\bbR$ satisfying $0\leq t_1<t_2\leq \tau$ we define
    \begin{align*}
        \zeta_{t_1, t_2}(t):=\begin{cases}
            1&\mbox{ if }t\in t\in[o,t_1]\\
            \frac{t_2-t}{t_2-t_1}&\mbox{ if }t\in(t_1,t_2),\\
            0&\mbox{ if }t\in[t_2,\tau].
        \end{cases} 
    \end{align*}
    We define the comparison map $v=u+\zeta_{t_1,t_2}([u]_h-u)$ to be used in the variational inequality \cref{varineq2}. The admissibility of $v$ as a comparison map can be argued similarly to how it was done earlier in \cref{localizesub}. We make use of the convexity of the functional $\FenR$ and \cref{convtvrfunc} to get
    \begin{align}\label{varineq3}
        \int_{0}^{t_2} \FenR(u(t)) \,dt \leq& \int_{0}^{t_2}\int_{\Omega}\zeta_{t_1,t_2}\partial_t u([u]_h-u) \,dx \,dt + \int_{0}^{t_2} (1-\zeta_{t_1,t_2})\FenR(u(t))+\zeta_{t_1,t_2}[\FenR(u(t))]_h\,dt.
    \end{align}
    This leads to the following chain of inequalities
    \begin{align*}
        0&\leq \int_0^{t_2}\int_\Omega \zeta_{t_1,t_2}\de_tu([u]_h-u)\,dx\,dt + \int_{0}^{t_2}\zeta_{t_1,t_2}[[\FenR(u)]_h-\FenR(u)]\,dt\\
        &\stackrel{\cref{ODEmolli}}{=}-h\int_0^{t_2}\int_\Omega \zeta_{t_1,t_2}\de_tu\de_t[u]_h\,dx\,dt - h\int_{0}^{t_2}\zeta_{t_1,t_2}\de_t[\FenR(u)]_h\,dt\\
        &= -h\int_0^{t_2}\int_\Omega \zeta_{t_1,t_2}\de_tu\de_t[u]_h\,dx\,dt + h\int_0^{t_2}\zeta_{t_1,t_2}^{'}[\FenR(u)]_h\,dt + h\FenR(u_0),
    \end{align*} where we integrate by parts in the last equation. Finally, divide both sides by $h$ and let $h\to 0$ to get
    \begin{align*}
        \frac{1}{t_2-t_1}\int_{t_1}^{t_2}\FenR(u(t))\,dt \leq \FenR(u_0) - \int_0^{t_2}\int_\Omega \zeta_{t_1,t_2}|\de_t u|^2\,dx\,dt\leq \FenR(u_0).
    \end{align*}
    This finishes the proof. The proof of \cref{timeregularity2} is similarly achieved with obvious modifications.
\end{proof}

\section{Uniqueness theorems}

\subsection{Proof of \cref{uniqueness1}} The uniqueness of the flow corresponding to $\tvr$ deblurring as defined in \cref{def:solr} follows from the strong convexity of the functional $\FenR$ which is a result of the injectivity of the operator $K$. We prove this result below.
\begin{lem}
    Suppose that $K$ is injective. Then for any $0<T < \infty$ and any initial data $u_0 \in L^2(\Omega) \cap \bvr(\Omega)$ there is at most one variational solution to the $\tvr$-deblurring flow in the sense of \cref{def:solr}. 
\end{lem}

\begin{proof}
    Suppose that $u_1,u_2\in \lwbv\cap C^0([0,T];L^2(\Omega))$ be two distinct variational solutions to the $\tvr$-deblurring flow in the sense of \cref{def:solr} with initial data $u_0$. We add the two variational inequalities in \cref{eq:sol} corresponding to $u_1$ and $u_2$. We drop the final-time term from the inequality since it is non-positive and on the RHS.
    \begin{align}
        \int_0^T \FenR(u_1(t))+\FenR(u_2(t))\,dt \leq 2 \int_0^T\int_\Omega \de_t v(v-w)\,dx + \FenR(v(t))\,dt + \|v(0)-u_0\|^2_{L^2(\Omega)},
    \end{align} where $w=(u_1+u_2)/2$. We can take $v=w$ as the comparison map. It is straightforward to check that the function $w$ is admissible. Observe that $w(0)=u_0$ so that we get
    \begin{align*}
        \frac12\int_0^T \FenR(u_1(t))+\FenR(u_2(t))\,dt\leq \FenR(w(t))\,dt < \frac12\int_0^T \FenR(u_1(t))+\FenR(u_2(t))\,dt,
    \end{align*} where the last inequality follows from the strict convexity of $\FenR$ which itself is a consequence of the injectivity of $K$. This leads to the desired contradiction. 
\end{proof}

\subsection{Proof of \cref{uniqueness2}}\label{sec:uni} The uniqueness for the variational solution for $\tvG$ denoising flow as in \cref{def:solg} is proved below. This stipulates that for boundary conditions specified in the complement of $\Omega$, i.e., on $(\bbr\setminus\Omega)\times[0,T]$, the solutions are uniquely determined by them. On the other hand, specifying a boundary data merely on the boundary of $\Omega$ does not make sense due to the absence of a suitable trace operator. The proof is taken from \cite[Theorem 3.3]{Zhou1992}, where a comparison principle is also proved for variational solutions. This method of uniqueness, which is valid for all monotone operators, already appears in \cite{Lions1969}.

\begin{proof}(Proof of \cref{uniqueness2})
    Let $u, \tilde{u}\in \lwbvGz$ be two solutions to the $\tvG$-denoising flow as in \cref{def:solg}. By time-regularity proved in \cref{timeregularity2}, we have $\de_t u, \de_t\tilde{u}\in L^2(\Omega_T)$.
   Therefore we can integrate by parts in \cref{eq:solg} to get
    \begin{align}\label{varineq6}
        \int_{0}^{\tau} \FenG(u(t)) \,dt \leq& \int_{0}^{\tau}\int_{\Omega}\partial_t u(v-u) \,dx \,dt + \int_{0}^{\tau} \FenG(v(t))\,dt
    \end{align} for any $\tau\in (0,T]$ and all admissible comparison maps $v$. Taking $v=\tilde{u}$ in \cref{varineq6}, we get
    \begin{align}\label{varineq7}
        \int_{0}^{\tau} \FenG(u(t)) \,dt \leq& \int_{0}^{\tau}\int_{\Omega}\partial_t u(\tilde{u}-u) \,dx \,dt + \int_{0}^{\tau} \FenG(\tilde{u}(t))\,dt.
    \end{align}
    Interchanging the role of $u$ and $\tilde{u}$, we get
    \begin{align}\label{varineq8}
        \int_{0}^{\tau} \FenG(\tilde{u}(t)) \,dt \leq& \int_{0}^{\tau}\int_{\Omega}\partial_t \tilde{u}(u-\tilde{u}) \,dx \,dt + \int_{0}^{\tau} \FenG({u}(t))\,dt.
    \end{align}
    Adding \cref{varineq7} and \cref{varineq8}, we obtain
    \begin{align*}
        0&\leq \int_{0}^{\tau}\int_{\Omega}\partial_t u(\tilde{u}-u) \,dx \,dt+\int_{0}^{\tau}\int_{\Omega}\partial_t \tilde{u}(u-\tilde{u}) \,dx \,dt\\
        &=-\int_0^\tau\int_\Omega \de_t (u-\tilde{u})(u-\tilde{u})\,dx\,dt\\
        &=-\frac12\int_0^\tau\int_\Omega \de_t (u-\tilde{u})^2\,dx\,dt.
    \end{align*}
    Integrating the last expression, we get
    \begin{align*}
        0\leq -\frac12\int_\Omega (u-\tilde{u})^2(\tau)\,dx+\frac12\int_\Omega (u-\tilde{u})^2(0)\,dx = -\frac12\int_\Omega (u-\tilde{u})^2(\tau)\,dx,
    \end{align*} on account of the initial condition. We are therefore left with
    \begin{align*}
        \int_\Omega (u-\tilde{u})^2(\tau)\,dx \leq 0
    \end{align*} for any $\tau\in (0,T)$. Hence, $u(x,t)=\tilde{u}(x,t)$ a.e. $(x,t) \in\Omega\times(0,T)$.
\end{proof}
\begin{rem}
    We note that the proof of \cref{uniqueness2} also works in the case of \cref{uniqueness1}. The proofs both use the time regularity that comes from existence theory which itself makes use of convexity (but not strict convexity.) It may be possible to provide a proof of uniqueness without using the time regularity (and hence a proof independent of the existence theory) by using time mollification similarly to \cite[Section 4]{Prasad2023}.
\end{rem}

\section{Existence theorems}
In this section, we will prove the existence theorem \cref{existence1} and indicate the modifications required for the proof of \cref{existence2} since both are similar.

\subsection{A minimising sequence} Let $0<T<\infty$. For $0<\varepsilon\leq 1$ let
\[
\mathcal{F}_{\varepsilon}(v) = \int_0^T \exp(-t/\varepsilon)\left(\frac{1}{2}\int_{\Omega}|\de_t v|^2 \,dx + \frac{1}{\varepsilon}\FenR(v(t)))\right)\,dt.
\]
Consider the following space of functions:
\[
\mathcal{K} = \left\{v \in \lwbv : \de_t v \in L^2(\Omega_T)\right\}
\]
with the norm
\[
\|v\|_{\mathcal{K}} = \int_0^T \|v(t)\|_{\bvr(\Omega)}\,dt + \|\de_t v\|_{L^2(\Omega_T)}
\]
By $\mathcal{K}_0$ we denote the subset of those $v \in \mathcal{K}$ such that $v(0) = u_0$ in the $L^2$ sense. We note that 
\[
\mathcal{F}_{\varepsilon}(v) < +\infty \quad \mbox{ for } \quad v \in \mathcal{K}_0
\]
because $v \in L^2(\Omega_T)$. We also have the following $L^1$ bound for functions in $\mathcal{K}_0$
\begin{align*}
    \|v(t)\|_{L^1(\Omega)} &\leq \|v(t)-u_0\|_{L^1(\Omega)}+\|u_0\|_{L^1(\Omega)}\\
    &\leq \int_\Omega\left|\int_0^t\de_\tau v(\tau)\,d\tau\right|\,dx+\|\|_{L^1(\Omega)}\\
    &\leq |\Omega_T|^{\frac{1}{2}}\|\de_t v\|_{L^2(\Omega_T)} + \|u_0\|_{L^1(\Omega)}
\end{align*}
which after integrating with respect to time implies that
\[
\|v\|_{L^1(\Omega_T)} \leq T\left(|\Omega_T|^{\frac{1}{2}}\|\de_t v\|_{L^2(\Omega_T)} + \|u_0\|_{L^1(\Omega)}\right).
\]
We begin by proving the existence of approximate minimisers.
\begin{lem}\label{lem:!min}
    For any $\varepsilon \in (0,1]$, the functional $\mathcal{F}_{\varepsilon}$ admits a unique minimiser $u_{\varepsilon} \in \mathcal{K}_0$.
\end{lem}
\begin{proof}
    Let $(u_n)_{n\in\mathbb{N}} \subseteq \mathcal{K}_0$ be a minimising sequence. Then without loss of generality we may assume that 
    \[
    \sup_{n\in\mathbb{N}}\mathcal{F}_{\varepsilon}(u_n) \leq \FenR(u_0) + 1.
    \]
    We seek uniform(-in-$n$) $L^1$ bounds to apply our compactness lemma. To this end, let $v \in \mathcal{K}_0$. Then 
    \begin{align*}
        \|v\|_{\mathcal{K}} &\leq \int_0^T \tvr(v(t))\,dt + \left(1+T|\Omega_T|^{\frac{1}{2}}\right)\|\de_t v\|_{L^2(\Omega_T)}+ T\|u_0\|_{L^1(\Omega)}\\
        &\leq \left(1+T|\Omega_T|^{\frac{1}{2}}\right)\left[\int_0^T \tvr(v(t))\,dt + \|\de_t v\|_{L^2(\Omega_T)}\right]+ T\|u_0\|_{L^1(\Omega)}\\
        &\leq 2\exp(T/\varepsilon)\left(1+T|\Omega_T|^{\frac{1}{2}}\right)\mathcal{F}_{\varepsilon}(v) +  T\|u_0\|_{L^1(\Omega)}.
    \end{align*}
    In particular, setting $v = u_n$ we obtain
    \begin{align*}
        \sup_{n \geq 1}\left(\int_0^T \|u_n(t)\|_{\bvr(\Omega)}\,dt+\|\de_t u_n\|_{L^2(\Omega_T)}\right)
        \leq 2\exp(T/\varepsilon)\left(1+T|\Omega_T|^{\frac{1}{2}}\right)(\FenR(u_0)+1) +  T\|u_0\|_{L^1(\Omega)}.
    \end{align*}
    Therefore by \cref{lem:com}, we obtain a subsequence, which we continue to denote by $u_n$, and a function $u:\Omega_T \rightarrow \bb{R}$ such that
   \begin{itemize}
       \item[(a)] $u_n \rightarrow u$ in $L^1(\Omega_T)$,
       \item[(b)] $u_n \rightarrow u$ a.e. in $\Omega_T$,
       \item[(c)] $u_n \rightharpoonup u$ in $L^2(\Omega_T)$-weak. 
       \item[(d)] $\de_t u_n \rightharpoonup \de_t u$ in $L^2(\Omega_T)$-weak.  
   \end{itemize}
    We note that by \cref{lem:closure1}, the limit $u \in \lwbv$. Let $0\leq s < t \leq T$. Then 
    \[
    \|u_n(t)-u_n(s)\|_{L^2(\Omega)} \leq \sqrt{|t-s|}\|\de_tu_n\|_{L^2(\Omega_T)} \leq \exp(T/\varepsilon)(\FenR(u_0)+1)\sqrt{|t-s|}.
    \]
    It follows that $u_n$ is uniformly bounded in $C^{0,\frac{1}{2}}([0,T];L^2(\Omega))$ and so by passing to a further subsequence if necessary we may assume that $u_n$ converges to $u$ in $C^{0}([0,T];L^2(\Omega))$. Thus, $u(0) = u_0$ in $L^2$. 
    Since $K$ is a bounded operator as in \cref{deblurfunc}, we get
    \[
    \lim_{n \rightarrow +\infty} \int_{\Omega}|K(u_n(t))-u_0|^2\,dx = \int_{\Omega}|K(u(t))-u_0|^2\,dx \quad \mbox{ for a.e. } \quad 0<t<T
    \]
    Finally, recalling that $\tvr$ is lower semicontinuous with respect to $L^1$ convergence (\cite[Proposition 2.6]{Antil2024}) we get that
    \[
    \FenR(u(t)) \leq \liminf_{n \rightarrow +\infty}\,\FenR(u_n(t)) \quad \mbox{ for a.e. } \quad 0<t<T.
    \]
    Thus,
    \begin{align*}
        \mathcal{F}_{\varepsilon}(u) &= \int_0^T \exp(-t/\varepsilon)\left(\frac{1}{2}\int_{\Omega}|\de_t u|^2 \,dx + \frac{1}{\varepsilon}\FenR(u(t)))\right)\,dt\\
        &\leq \liminf_{n \rightarrow +\infty}\int_0^T \exp(-t/\varepsilon)\left(\frac{1}{2}\int_{\Omega}|\de_t u_n|^2 \,dx + \frac{1}{\varepsilon}\FenR(u_n(t)))\right)\,dt\\
        &=\liminf_{n \rightarrow +\infty} \mathcal{F}_{\varepsilon}(u).
    \end{align*}
    It follows that $u$ is a minimiser of $\mathcal{F}_{\varepsilon}$ in $\mathcal{K}_0$. It is unique since $\mathcal{F}_{\varepsilon}$ is strictly convex. 
\end{proof}
\subsection{Reformulation of minimality} Fix $\varepsilon$ in $0<\varepsilon \leq 1$ and let $u_{\varepsilon}$ be the unique minimiser of $\mathcal{F}_{\varepsilon}$ in $\mathcal{K}_0$. Let $\Psi \in \lwbv$ with $\de_t \Psi \in L^2(\Omega_T)$ and $\Psi(0) \in L^2(\Omega_T)$. Suppose that $\Upsilon \in W^{1,\infty}(0,T;[0,1])$ and that $\Psi(0)\Upsilon(0) = 0$. For $0<\delta\leq \exp(-T/\varepsilon)$, let
\begin{align*}
v_{\delta}(x,t) = u_{\varepsilon}(x,t) + \delta \exp(t/\varepsilon)\Upsilon(t)\Psi(x,t).
\end{align*}
First, we show that $v_{\delta} \in \mathcal{K}_0$ is a valid comparison map in \cref{def:solr}. It is easy to verify that $v_\delta\in\lwbv$ by arguments similar to those in \cref{localizesub}. Particularly, $v_\delta$ is a convex combination of $u_\varepsilon$ and $u_\varepsilon + \Psi$ on each time slice which along with the convexity of $\tvr$-seminorm lets us conclude that $v_\delta\in\lwbv$. Next, we check the finiteness of $\mathcal{F}_\varepsilon(v_\delta)$ by noting that
\begin{align*}
    \int_0^T\exp\left(-t/\varepsilon\right)\FenR(v_\delta(t))+\,dt &\leq \int_0^T\exp\left(-t/\varepsilon\right)\left[(1-\sigma(t))\FenR(u_\varepsilon(t))+\sigma(t)\FenR((u_\varepsilon+\Psi)(t))\right]\,dt\\
    & \leq \int_0^T\exp\left(-t/\varepsilon\right)\FenR(u_\varepsilon(t))\,dt + \int_0^T\FenR((u_\varepsilon+\Psi)(t))\,dt<\infty,
\end{align*} where $\sigma(t)=\delta\exp(t/\varepsilon)\Upsilon(t)$. That $\de_t v_\delta\in L^2(\Omega_T)$ is a consequence of a similar argument. The initial and boundary conditions are verified through $\Psi(0)\Upsilon(0)=0$.
Finally, by minimality of $u_\varepsilon$, we have
\begin{align*}
    \mathcal{F}_\varepsilon(u_\varepsilon)\leq\mathcal{F}(v_\delta)<\infty.
\end{align*}
This implies that
\begin{align*}
&\int_0^T \exp(-T/\varepsilon)\int_{\Omega}\frac{\delta^2|\de_t(\exp(t/\varepsilon)\Upsilon\Psi)|^2}{2}+\delta\de_tu_{\varepsilon}\de_t(\exp(t/\varepsilon)\Upsilon\Psi) \,dx\,dt \\
&\qquad+ \int_0^T\frac{\delta}{\varepsilon}\Upsilon(t)(\FenR((u_{\varepsilon}+\phi)(t))-\FenR(u_{\varepsilon}(t)))\,dt \\
&\geq \int_0^T \exp(-T/\varepsilon)\int_{\Omega}\frac12 \left(|\de_tu_\varepsilon+\delta\de_t(\exp(t/\varepsilon)\Upsilon\Psi)|^2-|\de_tu_\varepsilon|^2\right) \,dx\,dt \\
&\qquad+ \int_0^T\frac{\exp(-t/\varepsilon)}{\varepsilon}(\FenR((u_{\varepsilon}+\delta\exp(t/\varepsilon)\Upsilon\Psi(t))-\FenR(u_{\varepsilon}(t)))\,dt\\
&\geq 0,
\end{align*} where the first inequality uses convexity just as in the previous displays.
Multiplying throughout by $\varepsilon/\delta$ and then sending $\delta \rightarrow 0+$ we first obtain
\begin{align*}
    0&\leq \int_0^T \exp(-T/\varepsilon)\int_{\Omega}\delta\de_tu_{\varepsilon}\de_t(\exp(t/\varepsilon)\Upsilon\Psi) \,dx\,dt \\
    &\qquad+ \int_0^T\Upsilon(t)(\FenR((u_{\varepsilon}+\Psi)(t))-\FenR(u_{\varepsilon}(t)))\,dt\\    &=\int_0^T\Upsilon(t)\left[\de_tu_{\varepsilon}\Psi+(\FenR((u_{\varepsilon}+\Psi)(t))-\FenR(u_{\varepsilon}(t)))\right]\,dt\\
    &\qquad+\varepsilon\int_0^T\int_\Omega\left[\Upsilon'\de_t u_\varepsilon\Psi + \Upsilon\de_tu_\varepsilon\de_t\Psi\right]\,dx\,dt
\end{align*} which can be rewritten to
\begin{gather}\label{eq:reformulation}
\begin{aligned}
    \int_0^T \Upsilon(t)\FenR(u_{\varepsilon}(t))\,dt &\leq \int_0^T \Upsilon(t)\FenR((u_{\varepsilon}+\Psi)(t))\,dt\\
    &\quad+ \int_0^T\int_{\Omega}\Upsilon\de_tu_{\varepsilon}\Psi \,dx\,dt\\
    &\quad\quad+\varepsilon\int_0^T\int_{\Omega}\left(\Upsilon'\de_tu_{\varepsilon}\Psi + \Upsilon\de_tu_{\varepsilon}\de_t{\Psi}\right)\,dx\,dt.
    \end{aligned}
\end{gather}
This is a new reformulation of the minimality condition for the approximate variational solutions and it is valid for all $\Upsilon\in W^{1,\infty}(0,T;[0,1])$ and $\Psi\in \lwbv$ with $\de_t\Psi\in L^2(\Omega_T)$ satisfying $\Upsilon(0)\Psi(0)=0$ in the sense of $L^2(\Omega)$.

\subsection{Energy bounds} In order to pass to the limit in the minimality condition as $\varepsilon\to 0$, we need energy bounds that are uniform in $\varepsilon$. These will be proved in this section. We choose $\Psi = -h\de_t[u_{\varepsilon}]_h$ in \eqref{eq:reformulation} to get
\begin{align*}
h\int_0^T\int_{\Omega}(\Upsilon+\varepsilon\Upsilon')\de_tu_{\varepsilon}\de_t[u_{\varepsilon}]_h + \varepsilon\Upsilon\de_tu_{\varepsilon}\de_{tt}[u_{\varepsilon}]_h \,dx\,dt &\leq \int_0^T\Upsilon(t)\left[\FenR([u_\varepsilon]_h(t))-\FenR(u_\varepsilon(t))\right]\,dt\\
&\leq \int_0^T \Upsilon(t)\left[[\FenR(u_\varepsilon(t))]_h-\FenR(u_\varepsilon(t))\right]\,dt\\
&\stackrel{\cref{ODEmolli}}{=}-h\int_0^T\Upsilon(t)\de_t[\FenR(u_{\varepsilon}(t))]_h\,dt.
\end{align*} We also make use of \cref{convfuncR}. 
Dividing the previous inequality by $h$ and noting that
\begin{align*}
\de_tu_{\varepsilon}\de_{tt}[u_{\varepsilon}]_h &=\de_t[u_\varepsilon]_h\de_{tt}[u_\varepsilon]_h + (\de_tu_\varepsilon-\de_t[u_\varepsilon]_h)\de_{tt}[u_\varepsilon]_h\\
&=\frac12\de_t|\de_t[u_\varepsilon]_h|^2+\frac1h|\de_t[u_\varepsilon]_h-\de_tu_\varepsilon|^2\\
&\geq \frac{1}{2}\de_t|\de_t[u_{\varepsilon}]_h|^2
\end{align*}
we get
\begin{equation}\label{eq:energy-min}
  \int_0^T\int_{\Omega}(\Upsilon+\varepsilon\Upsilon')\de_tu_{\varepsilon}\de_t[u_{\varepsilon}]_h + \frac{\varepsilon}{2}\de_t|\de_t[u_{\varepsilon}]_h|^2 \,dx\,dt \leq -h\int_0^T\Upsilon(t)\de_t[\FenR(u_{\varepsilon}(t))]_h\,dt.
\end{equation}
We will use this inequality to obtain the required bounds.
\subsubsection{Uniform $L^2$ and $C^{0,\tfrac{1}{2}}$ bounds} We put $\Upsilon = 1$ in \eqref{eq:energy-min} to obtain
\begin{align*}
    \int_0^T\int_\Omega\de_tu_\varepsilon\de_t[u_\varepsilon]_h\,dx\,dt \leq -\int_0^T \de_t[\FenR(u_\varepsilon(t))]_h\,dt - \frac12\varepsilon\int_0^T\int_\Omega\de_t|\de_t[u_\varepsilon]_h|^2\,dx\,dt\leq \FenR(u_0),
\end{align*}
where we make use of $[\FenR(u_\varepsilon)]_h(0)=\FenR(u_0)$, $\de_t[u_\varepsilon]_h(0)=0$, $[\FenR(u_\varepsilon)]_h(T)\geq 0$, and $|\de_t[u_\varepsilon]_h|^2(T)\geq 0$. We send $h \rightarrow 0+$ to get the uniform bound:
\[
\int_0^T \int_{\Omega} |\de_t u_{\varepsilon}|^2 \,dx\,dt \leq \FenR(u_0)
\]
which implies the following estimates:
\[
\|u_{\varepsilon}\|^2_{L^2(\Omega_T)}\leq 4T^2\|\de_tu_\varepsilon\|^2_{L^2(\Omega_T)}+2T\|u_0\|^2_{L^2(\Omega)} \leq 4T^2\FenR(u_0) + 2T\|u_0\|_{L^2(\Omega)}^2
\]
and
\begin{equation}\label{eq:uniformC0}
    \|u_{\varepsilon}(t)-u_{\varepsilon}(s)\|_{L^2(\Omega)}\leq \|\de_tu_\varepsilon\|_{L^2(\Omega_T)}\sqrt{|t-s|} \leq \sqrt{\FenR(u_0)}\sqrt{t-s}
\end{equation}
for any $0 \leq s < t \leq T$. It follows that the sequence of minimisers $\{u_{\varepsilon}\}$ is uniformly bounded in $L^2(\Omega_T)$ and $C^{0,\tfrac{1}{2}}([0,T];L^2(\Omega))$.

\subsubsection{Uniform $\tvr$ bounds} Next, for $0 \leq t<s<T$ we put
\begin{align*}
\Upsilon(t) = \begin{cases}
    1&\mbox{ for }t\in[0,t_1],\\
    \frac{t_2-t}{t_2-t_1}&\mbox{ for }t\in(t_1,t_2),\\
    0&\mbox{ for }t\in[t_2,T]
\end{cases}
\end{align*}
in \eqref{eq:energy-min} and integrate by parts to get
\begin{align*}
    \int_0^T\int_{\Omega} \Upsilon(t)\de_tu_{\varepsilon}\de_t[u_{\varepsilon}]_h\,dx\,dt &\leq \FenR(u_0) + \int_0^T \Upsilon'(t)|\de_tu_{\varepsilon}|^2[\FenR(u_{\varepsilon}(t))]_h\,dt\\
    &+\int_0^T \Upsilon'(t)\frac{\varepsilon|\de_t[u_{\varepsilon}]_h|^2}{2} - \varepsilon\de_tu_{\varepsilon}\de_t[u_{\varepsilon}]_h \,dx\,dt.
    \end{align*}
We send $h \rightarrow 0+$ and find that
\[
  \int_0^T\int_{\Omega} \Upsilon(t)\de_t|u_{\varepsilon}|^2\,dx\,dt \leq \FenR(u_0) + \int_0^T \Upsilon'(t)|\de_tu_{\varepsilon}|^2\FenR(u_{\varepsilon}(t))\,dt\\
    -\int_0^T \Upsilon'(t)\frac{\varepsilon|\de_t u_{\varepsilon}|^2}{2}.
\]
is true. It follows that
\begin{equation}\label{eq:uniform-F}
    \int_{t_1}^{t_2}\FenR(u_{\varepsilon}(t))\,dt\leq (t_2-t_1)\FenR(u_0)+\frac{\varepsilon}{2}\int_{t_1}^{t_2}\int_\Omega |\de_tu_\varepsilon|^2\,dx\,dt \leq (t_2-t_1+\varepsilon/2)\FenR(u_0).
\end{equation}
In particular,
\begin{align*}
\int_{t_1}^{t_2} \tvr(u_{\varepsilon}(t)) \,dt \leq (t_2-t_1+\varepsilon/2)\FenR(u_0).
\end{align*}
\subsection{Limits} Using the previous uniform bounds we can invoke \cref{lem:com} to find a function $u: \Omega_T \rightarrow \bbr$ such that upto a subsequence, still denoted by $u_{\varepsilon}$, we have 
\begin{itemize}
    \item[(a)] $u_{\varepsilon} \rightarrow u  \mbox{ strongly in }  L^1(\Omega_T)$
    \item[(b)] $u_{\varepsilon} \rightarrow u  \mbox{ a.e. on }  \Omega_T$
    \item[(c)] $u_{\varepsilon} \rightharpoonup u \mbox{ and } \de_t u_{\varepsilon} \rightharpoonup \de_tu \mbox{ in }  L^2(\Omega_T)$-weak.
\end{itemize}
Not that the uniform $\tvr$ bound and the strong $L^1$ convergence imply that $u \in \lwbv$ by \cref{lem:closure1}. Since $\de_t u_{\varepsilon} \rightharpoonup \de_tu$ in $L^2$ we have
\[
\int_0^T \int_{\Omega} |\de_t u|^2 \,dx\,dt \leq \liminf_{\varepsilon \rightarrow 0+} \int_0^T\int_{\Omega}|\de_t u_{\varepsilon}|^2 \,dx\,dt \leq \FenR(u_0). 
\]
Since $\FenR$ is lower semi-continuous on each time slice - for the $\tvr$ part we again appeal to lower semicontinuity with respect to $L^1$ convergence and for the deblurring part we appeal to the weak convergence in $L^2$ again as well as Fatou's lemma, we get 
\[
\int_{t_1}^{t_2} \FenR(u(t)) \,dt \leq \int_{t_1}^{t_2} \liminf_{\varepsilon \rightarrow 0+} \FenR(u_{\varepsilon}(t)) \,dt \leq \liminf_{\varepsilon \rightarrow 0+} \int_{t_1}^{t_2} \FenR(u_{\varepsilon}(t)) \,dt \leq (t_2-t_1)\FenR(u_0) < +\infty. 
\]
Setting $t_2=T$ and $t_1=0$ we get
\[
0 \leq \int_0^T \FenR(u(t)) \,dt < + \infty.
\]
This makes the LHS of \cref{eq:sol} finite. Again as in \cref{lem:!min}, we also get that $u(0) = u_0$ in $L^2$. We will now show that $u$ is a variational solution. To this end, consider $v \in \lwbv$ with $\de_t v \in L^2(\Omega_T)$ , $v(0) \in L^2(\Omega)$ and 
\[
\int_0^T \FenR(v(t)) \,dt < +\infty.
\]  We will again use \eqref{eq:reformulation} with 
\begin{align}
\Upsilon(t) = \begin{cases}
    \frac{t}{\theta}&\mbox{ for }t\in[0,\theta),\\
    1&\mbox{ for }t\in[\theta,T-\theta],\\
    \frac{T-t}{\theta}&\mbox{ in }t\in(T-\theta,t].
\end{cases}
\end{align}
where $\theta \in (0,T/2)$. For $\Psi$, we fix an $\varepsilon \in (0,1]$ and set
\[
\Psi = v-u_{\varepsilon}.
\]
These functions satisfy the requirement for the reformulated minimality condition \cref{eq:reformulation}. Doing so we obtain
\begin{align*}
    \int_0^T \FenR(u_{\varepsilon}(t)))\,dt \leq& \int_0^T (1-\Upsilon(t))\FenR(u_{\varepsilon}(t))\,dt + \int_0^T\int_{\Omega}\Upsilon\de_tu_{\varepsilon}(v-u_{\varepsilon})\,dx\,dt\\
    &+\int_0^T\Upsilon(t)\FenR(v(t)))\,dt + \varepsilon\int_0^T\int_{\Omega}\left(\Upsilon'\de_tu_{\varepsilon}(v-u_{\varepsilon}) + \Upsilon\de_tu_{\varepsilon}\de_t{(v-u_{\varepsilon})}\right)\,dx\,dt.\\
    &= \mathfrak{A} +\mathfrak{B} + \mathfrak{C} + \mathfrak{D}.
\end{align*}
For $\theta \geq \varepsilon$ we use \eqref{eq:uniform-F} to get
\[
\mathfrak{A}\leq \int_0^\theta \FenR(u_{\varepsilon}(t))\,dt+\int_{T-\theta}^T \FenR(u_{\varepsilon}(t))\,dt \leq (2\theta + \varepsilon)\FenR(u_0) \leq 3\theta \FenR(u_0).
\]
Next we write:
\[
\mathfrak{B} = \int_0^T\int_{\Omega}\Upsilon\de_tv(v-u_{\varepsilon})\,dx\,dt - \frac{1}{2}\int_0^T\int_{\Omega}\Upsilon\de_t|v-u_{\varepsilon}|^2\,dx\,dt.
\]
Integrating by parts while noting that $\Upsilon(T)=\Upsilon(0)=0$, we get
\[
 - \frac{1}{2}\int_0^T\int_{\Omega}\Upsilon\de_t|v-u_{\varepsilon}|^2\,dx\,dt =\frac{1}{2}\int_0^T\int_{\Omega}\Upsilon'|v-u_{\varepsilon}|^2\,dx\,dt= \frac{1}{2\theta}\int_0^{\theta}\int_{\Omega}|v-u_{\varepsilon}|^2\,dx\,dt - \frac{1}{2\theta}\int_{T-\theta}^T|v-u_{\varepsilon}|^2\,dx\,dt. 
\]
On the other hand using Minkowski inequality along with \eqref{eq:uniformC0} we get
\begin{align*}
    \frac{1}{2\theta}\int_{\Omega}|v-u_{\varepsilon}|^2 \,dx\,dt&\leq \left\{\left(\frac{1}{2\theta}\int_{\Omega}|v-u_{0}|^2 \,dx\,dt+\right)^{\frac{1}{2}}+\left(\frac{1}{2\theta}\int_{\Omega}|u_\varepsilon-u_{0}|^2 \,dx\,dt\right)^{\frac12}\right\}^{2}\\
    &\leq \left\{\left(\frac{1}{2\theta}\int_{\Omega}|v-u_{0}|^2 \,dx\,dt\right)^{\frac{1}{2}}+\left(\frac{\theta}{2}\FenR(u_0)\right)^{\frac{1}{2}}\right\}^{2}.
\end{align*}
Thus invoking weak convergence in $L^2$ we send $\varepsilon \rightarrow 0+$ to get
\begin{align*}
\liminf_{\varepsilon \rightarrow 0+} \mbox{ } \mathfrak{B} \leq& \int_0^T\int_{\Omega}\Upsilon\de_tv(v-u)\,dx\,dt  - \frac{1}{2\theta}\int_{T-\theta}^T|v-u|^2\,dx\,dt \\
&+ \left\{\left(\frac{1}{2\theta}\int_{\Omega}|v-u_{0}|^2 \,dx\,dt\right)^{\frac{1}{2}}+\left(\frac{\theta}{2}\FenR(u_0)\right)^{\frac{1}{2}}\right\}^{2}.
\end{align*}
Finally
\[
\mathfrak{D} \rightarrow 0 \quad \mbox{ as } \quad \varepsilon \rightarrow 0+
\]
due to the uniform $L^2$ bounds on $u_{\varepsilon}$ and $\de_t u_{\varepsilon}$. We now invoke lower semicontinuity of $\FenR$ once again to get
\begin{align*}
   \int_0^T \FenR(u(t))\,dt \leq& \int_0^T \Upsilon(t)\int_{\Omega} \de_t v (v-u)\,dx\,dt + \int_0^T \Upsilon(t)\FenR(v(t))\,dt + 3\theta \FenR(u_0) \\
   &+   \left\{\left(\frac{1}{2\theta}\int_{\Omega}|v-u_{0}|^2 \,dx\,dt\right)^{\frac{1}{2}}+\left(\frac{\theta}{2}\FenR(u_0)\right)^{\frac{1}{2}}\right\}^{2} - \frac{1}{2\theta}\int_{T-\theta}^T|v-u|^2\,dx\,dt . 
\end{align*}
Sending $\theta \rightarrow 0+$ we get
\begin{align*}
      \int_0^T \FenR(u(t))\,dt \leq& \int_0^T\int_{\Omega} \de_t v (v-u)\,dx\,dt + \int_0^T \FenR(v(t))\,dt\\
   &+\frac{1}{2}\|v(0)-u_0\|_{L^2(\Omega)}^2 -\frac{1}{2}\|(v-u)(T)\|_{L^2(\Omega)}^2.
\end{align*}
It follows that $u$ is a variational solution. This completes the proof of \cref{existence1}.

\begin{rem}
    In this remark, we shall note the modifications that are required in order to prove \cref{existence2}. 
    
    Firstly, we define the functional $\mathcal{F}_\varepsilon$ space $\mathcal{K}_0$ in the same way. However, instead of seeking solutions in the class $\mathcal{K}_{0}$, we need the further class
    \begin{align*}
        \mathcal{K}_{\varepsilon,0}:=\{u\in\mathcal{K}_0:\mathcal{F}_\varepsilon(u)<\infty\}.
    \end{align*} Unlike the case of $\tvr$ deblurring flow, the finiteness of the $\mathcal{F}_\varepsilon$ functional is not automatic. Hence, at each step when a comparison map is introduced, one additonally needs to verify that it has finite $\mathcal{F}_\varepsilon$ value. 

    In the proof of \cref{lem:!min}, we also need to verify the boundary condition for the minimiser $u$ which follows from the almost everywhere $(x,t)\in\bbr\times(0,T)$ convergence of the sequence $u_n$. The same argument also allows for the limit of $u_\varepsilon$ as $\varepsilon\to 0$ to satisfy the boundary condition. 
    
    In the formulation of the minimality condition \cref{eq:reformulation}, one additionaly assumes that $\int_0^T\FenR((u+\Psi)(t))\,dt<\infty$ in order for the function $v_\delta$ to be an admissible test function and for it to have finite $\mathcal{F}_\varepsilon$ value. 
    
    In the part of the proof where one proves the uniform energy bounds, for the choice of $\Psi=-h\de_t[u_\varepsilon]_h$, we need to prove that $\int_0^T\FenG((u+\Psi)(t))\,dt<\infty$. We first note that $[u_\varepsilon]_h\in \lwbvGz$ by \cref{convtvGfunc}. We also have $[u_\varepsilon]_h(0)=u_0$ and therefore $\de_t[u_\varepsilon]_h(0)=\frac1h(u_0-[u_\varepsilon]_h(0))=0$ and from \cref{convfuncG} we have
    \begin{align*}
        \int_0^T\Upsilon(t)\FenG((u_\varepsilon-h\de_t[u_\varepsilon]_h)(t))\,dt= \int_0^T\Upsilon(t)\FenG([u_\varepsilon]_h(t))\,dt<\infty
    \end{align*} for any $\Upsilon\in W^{1,\infty}(0,T;[0,1])$. This makes $\Psi=-h\de_t[u_\varepsilon]_h$ admissible in \cref{eq:reformulation}.  

    Apart from these, we replace the occurrences of $\FenR$ with $\FenG$ and those of $\tvr$ with $\tvG$, those of $\lwbv$ with $\lwbvGz$ etc. Finally, we invoke the relevant compactness lemma which is \cref{lem:com2}.
\end{rem}

\section{Variational solutions are parabolic minimisers}
Another variational notion of solutions to parabolic equations appears in the work of Wieser \cite{Wieser1987} in the form of parabolic minimisers.
\begin{defn}\label{parmindef}
    A measurable function $u:\Omega_T\to\bbR$ is a parabolic minimiser to the $\tvr$ deblurring flow if one has
    \begin{align*}
        u\in \lwbv\mbox{ and }\int_0^T\tvr(u(t))\,dt<\infty,
    \end{align*} and the following minimality condition holds
    \begin{align}\label{minimacond}
        \int_0^T\left(\int_\Omega u\cdot\de_t\varphi\,dx+\tvr(u(t))\right)\,dt\leq \int_0^T\tvr((u+\varphi)(t))\,dt,
    \end{align} for all $\varphi\in C_c^\infty(\Omega_T).$
\end{defn}

\begin{prop}
    Let $u$ be a variational solution to the $\tvr$ deblurring flow in the sense of \cref{def:solr}. Then, $u$ is a parabolic minimiser to the $\tvr$ deblurring flow in the sense of \cref{parmindef}. 
\end{prop}

\begin{proof}
    By \cref{timeregularity1}, we have $\de_t u\in L^2(\Omega)$. In \cref{eq:sol}, we take $v=u+s\varphi$ for $\varphi\in C_c^\infty(\Omega_T)$ and $s>0$. It is easy to check that is an admissible map. We also have $v(0)=u_0$ and $v(T)=u(T)$ due to the compact support of $\varphi$. As a result, \cref{eq:sol} becomes
    \begin{align*}
        \int_0^T\tvr(u(t))\,dt\leq \int_0^T\left(\int_\Omega \de_t(u+s\varphi)s\varphi\,dx + \tvr((u+s\varphi)(t))\right)\,dt
    \end{align*}
    We integrate by parts in the first term on the RHS. We do not get any terms on the boundary due to the compact support of $\varphi$. We also make use of the convexity of $\tvr$ to obtain
    \begin{align*}
        \int_0^t\left(\int_\Omega s\de_t\varphi(u+s\varphi)\,dx+\tvr(u(t))\right)\,dt\leq \int_0^T(1-s)\tvr(u(t)) + s\tvr((u+\varphi)(t))\,dt.
    \end{align*} We move the first term on the RHS to the left and divide by $s$ to get
    \begin{align*}
        \int_0^t\left(\int_\Omega \de_t\varphi(u+s\varphi)\,dx+\tvr(u(t))\right)\,dt\leq \tvr((u+\varphi)(t))\,dt.
    \end{align*} As soon as we pass to the limit in this inequality as $s\to 0$, we obtain \cref{minimacond}.
\end{proof}
\begin{rem}
    A similar result holds for the $\tvG$ denoising flow. In particular, we have 
    \begin{align}\label{minimacond2}
        \int_0^T\left(\int_\Omega u\cdot\de_t\varphi\,dx+\tvG(u(t))\right)\,dt\leq \int_0^T\tvG((u+\varphi)(t))\,dt,
    \end{align} for $u$ as in \cref{def:solg}. Since $\tvG(u) = [u]_{W^{\alpha,1}}$, we have
    \begin{align}\label{minimacond3}
        \int_0^T\left(\int_\Omega u\cdot\de_t\varphi\,dx+\int_\bbr\int_\bbr\frac{|u(x,t)-u(y,t)|}{|x-y|^{N+\alpha}}\,dy\,dx\right)\,dt\leq \int_0^T\int_\bbr\int_\bbr\frac{|(u+\varphi)(x,t)-(u+\varphi)(y,t)|}{|x-y|^{N+\alpha}}\,dt.
    \end{align} This furnishes another proof of the existence of solutions for the parabolic fractional $1$-Laplacian. See \cite{Dingding24} for another method in this direction. Our focus here is to describe fractional $\textbf{BV}$ image restoration methods.
\end{rem}

\section{Compactness lemmas} We collect here the compactness lemmas pertaining to the various stationary and evolutionary nonlocal BV spaces that are used in the article.

\subsection{Compactness in \texorpdfstring{$\bvr$}{BVR}} We begin by stating a compactness result for $\bvr$ whose proof may be found in \cite[Proposition 2.8]{Antil2024}.

\begin{lem}
    \label{lem:com00}
    Let $\Omega$ be an open and bounded set in $\bbr$. Assume $(f_k)_{k=1}^\infty\subset \bvr(\Omega)$ such that
    \begin{align*}
       \sup_{k\geq 1} \left\{\|f_k\|_{L^1(\Omega)} + \tvr(f_k,\Omega)\right\} < \infty.
    \end{align*} Then there exists $f\in \bvr(\Omega)$ such that 
    \begin{align*}
        \tvr(f;\Omega)\leq \liminf_{k\geq 1}\, \tvr(f_k;\Omega).
    \end{align*} Moreover, there is a subsequence $\{f_{k_i}\}_{i=1}^\infty$ such that $\|f_{k_i}-f\|_{L^p(\Omega)}\to 0$ as $i\to\infty$.
\end{lem}

\subsection{Compactness in \texorpdfstring{$\bvG$}{bvg}} We also have the following compactness result for functions in the Gagliardo type BV spaces, which are fixed, outside a bounded set.

\begin{lem}
    \label{lem:com01}
    Let $\Omega$ be an open and bounded set in $\bbr$. Assume $(f_k)_{k=1}^\infty\subset \bvG(\Omega;u_0)$ such that
    \begin{align*}
       \sup_{k\geq 1} \left\{\|f_k\|_{L^1(\Omega)} + \tvG(f_k,\bbr)\right\} < \infty.
    \end{align*} Then there exists $f\in \bvG(\Omega,u_0)$ such that 
    \begin{align*}
        \tvG(f;\bbr)\leq \liminf_{k\geq 1}\, \tvG(f_k;\bbr).
    \end{align*} Moreover, there is a subsequence $\{f_{k_i}\}_{i=1}^\infty$ such that $\|f_{k_i}-f\|_{L^p(\bbr)}\to 0$ as $i\to\infty$ for $1\leq p<\frac{N}{N-\alpha}$.
\end{lem}

\begin{proof}
    By \cite[Theorem 3.4]{Antil2024},
    \begin{align}
        \tvG(f_k;\bbr) = [f_k]_{W^{\alpha,1}(\bbr)}.
    \end{align}
    Set $g_k:=f_k-u_0$. Then, $g_k\in \bvG(\bbr)$ such that $g_k\equiv 0$ on $\bbr\setminus\Omega$ a.e. Choose a ball $B=B(0,r)$ such that $\Omega\subset B(0,r)$. To employ the compactness theorems of $W^{\alpha,1}(B)$, we must prove bounds on $\|g_k\|_{W^{\alpha,1}(B)}$, however
    \begin{align*}
        \|g_k\|_{W^{\alpha,1}(B)} = \|g_k\|_{L^{1}(B)}+[g_k]_{W^{\alpha,1}(B)}&\leq \|g_k\|_{L^{1}(\bbr)}+[g_k]_{W^{\alpha,1}(\bbr)}\\
        & \leq \|f_k\|_{L^{1}(\bbr)}+[f_k]_{W^{\alpha,1}(\bbr)}+ \|u_0\|_{L^{1}(\bbr)}+[u_0]_{W^{\alpha,1}(\bbr)}\\
        & = \|f_k\|_{L^{1}(\bbr)}+\tvG(f_k,\bbr)+ \|u_0\|_{L^{1}(\bbr)}+\tvG(u_0,\bbr),
    \end{align*} the last of which is bounded independently of $k$. A ball is an extension domain and therefore by \cite[Corollary 7.2]{DiNezza2012}, there is a subsequence $g_{k_i}$ and $g\in W^{\alpha,1}(B)$ with the properties:
    \begin{align*}
        \begin{cases}
            g_{k_i}\rightharpoonup g\mbox{ in } W^{\alpha,1}(B)\mbox{-weak},\\
            g_{k_i} \to g\mbox{ in }L^p(B), 1\leq p< \frac{N}{N-\alpha},\\
            g_{k_i}\to g\mbox{ a.e }x\in B.
        \end{cases}
    \end{align*}
    Now extend $g$ by zero outside $B$ and set $f:=g+u_0$.
    Clearly, for $p\in \left[1,\frac{N}{N-\alpha}\right)$ $$\|f_{k_i}-f\|_{L^p(\bbr)}=\|g_{k_i}-g\|_{L^p(\Omega)}\leq \|g_{k_i}-g\|_{L^p(B)}\to 0$$ as $i\to \infty$. Through an application of the lowersemicontinuity result (\cite[Lemma 3.8]{Antil2024}), one then gets $$\tvG(f;\bbr)\leq \liminf_{k\geq 1}\, \tvG(f_k;\bbr).$$ This completes the proof. (This technique, while elementary, came to the attention of the authors in \cite{Byun2022}.)
\end{proof}

\subsection{Compactness in \texorpdfstring{$\lwbv$}{lwbv}} The following theorem will be required for proving convergence of approximate solutions to the variational solution for the Riesz type nonlocal BV spaces.
\begin{lem}\label{lem:com}
    Let $\Omega\subset\bbr$ be a bounded open set. Suppose that
    \[
    u_n \in \lwbv \quad \mbox{ for all } \quad n \geq 1.
    \]
    If 
    \[
        \sup_{n \geq 1}\left(\int_0^T \tvr(u_n(t))\,dt+\|\de_t u_n\|_{L^1(\Omega_T)}+\|u_n\|_{L^1(\Omega_T)}\right) < +\infty 
    \]
    then 
    \[
    \mbox{the sequence } \{u_n\}_{n \geq 1} \mbox{ is precompact in } L^1(\Omega_T).
    \]
\end{lem}
\begin{proof}
    Let $0<h<T$ and consider the Steklov averages:
    \[
    (u_n)_h(t) = \frac{1}{h}\int_t^{t+h}u_n(s)\,ds \quad \mbox{ for } \quad 0\leq t \leq T-h.
    \]
    Then $(u_n)_h \in C^0([0,T-h];L^1(\Omega))$ and for any $0\leq t_1<t_2\leq T-h$ there holds 
    \[
    h\|(u_n)_h(t_2) - (u_n)_h(t_1)\|_{L^1(\Omega)} \leq (t_2-t_1)\|\de_t u_n\|_{L^1(\Omega_T)}.
    \]
    We fix an $h \in (0,T)$. Then the preceding estimate shows that $(u_n)_h$ is equicontinuous in $C^0([0,T-h];L^1(\Omega))$. Next we will show that $(u_n(t))_h$ with $0\leq t \leq T-h$ is precompact in $L^1(\Omega)$. This will put us in the setting of Arzela- Ascoli for Banach valued mappings. Indeed, let $\Psi \in C_c^1(\bbr,\bbr)$ with $\|\Psi\|_{\infty} \leq 1$. We compute
    \begin{align*}
        \int_{\Omega}(u_n)_h(t)\divr \Psi \,dx \leq \frac{1}{h}\int_{t}^{t+h}\tvr(u_n(s))\,ds \leq \frac{1}{h}\int_{0}^{T}\tvr(u_n(s))\,ds 
    \end{align*}
    which implies 
    \[
    \sup_{n \geq 1}\tvr((u_n)_h(t)) \leq \frac{1}{h} \sup_{n \geq 1}\int_0^T\tvr(u_n(s))\,ds < +\infty.
    \]
    Since we also have 
    \[
     \sup_{n \geq 1} \|u_n\|_{L^1(\Omega_T)}  < +\infty,
    \]
    it follows that $(u_n(t))_h$ is uniformly bounded in $\bvr(\Omega)$. Thus by \cref{lem:com00} it follows that $(u_n(t))_h$ is precompact in $L^1(\Omega)$ (note that we are always working with an extension by $0$ to the whole of $\bbr$.) Therefore by a Arzela-Ascoli type theorem for Banach valued maps (see \cite[Chapter 3, Theorem 3.1]{Lang1993}), it holds that the sequence $(u_n(t))_h$ is precompact in $C^0([0,T-h];L^1(\Omega))$ and hence in $L^1(0,T-h;L^1(\Omega))$. We fix an $S \in (0,T)$. We compute
    \[
    \|(u_n)_h-u_n\|_{L^1(0,S;L^1(\Omega))} \leq h\|\de_t u_n\|_{L^1(\Omega_T)}.
    \]
    Then the uniform bound on $\|\de_t u_n\|_{L^1(\Omega_T)}$ implies that
    \[
    \lim_{h \searrow 0}\sup_{n \geq 1}\|(u_n)_h - u_n\|_{L^1(0,S;L^1(\Omega))} = 0.
    \]
    Since the averages $(u_n)_h$ are precompact in $L^1(0,S;L^1(\Omega))$ for any $0<h<T-S$, it follows that $(u_n)_{n \geq 1}$ is precompact in $L^1(0,S;L^1(\Omega))$ for every $0<S<T$. By translating the solution appropriately and noting that the uniform bounds continue to hold for the translations, we conclude that the sequence $(u_n)_{n \geq 1}$ is precompact in $L^1(0,S;L^1(\Omega))$. The conclusion follows. 
\end{proof}

\subsection{Compactness in \texorpdfstring{$\lwbvGz$}{lwbvgz}} The following theorem will be used for proving convergence of approximate solutions to the variational solution in the setting of Gagliardo type nonlocal BV spaces. The proof is similar to \cref{lem:com}, however we repeat the arguments for clarity.

\begin{lem}\label{lem:com2}
    Let $\Omega\subset\bbr$ be a bounded open set. Suppose that $u_n \in \lwbvGz$ for all $n \geq 1$.
    \[
        \text{If }\sup_{n \geq 1}\left(\int_0^T \tvG(u_n(t);\bbr)\,dt+\|\de_t u_n\|_{L^1(\bbr\times [0,T])}+\|u_n\|_{L^1(\bbr\times[0,T])}\right) < +\infty 
    \] then the sequence $\{u_n\}_{n \geq 1}$ is precompact in $L^1(\bbr\times[0,T]).$
\end{lem}

\begin{proof}
    For $h\in (0,T)$, we use the Steklov time-mollification $(u_n)_h$ of $u_n$ 
    \begin{align*}
        (u_n)_h(t):=\frac1h\int_t^{t+h}u_n(s)\,ds\,\mbox{ for }t\in[0,T-h].
    \end{align*} By, we have $(u_n)_h\in C^0([0,T-h];L^1(\Omega))$ by \cite[Lemma 2.2]{chagas2017propertiessteklovaverages} and we can calculate as 
    \begin{align*}
        \|(u_n)_h(t_2)-(u_n)_h(t_1)\|_{L^1(\Omega)} & = \|\int_{t_1}^{t_2} \de_s(u_n)_h(s)\,dx\|_{L^1(\Omega)}\\
        & = \|\frac1h\int_{t_1}^{t_2}[u_n(s+h)-u_j(s)]\,ds\|_{L^1(\Omega)}\\
        &\leq \frac1h \int_{t_1}^{t_2}\int_s^{s+h}\|\de_\tau u_n(\tau)\|_{L^1(\Omega)}\,d\tau\,ds\\
        &\leq (t_2-t_1)\|\de_tu_n\|_{L^1(\Omega},
    \end{align*} where in the second last step, we use \cite[Lemma 4.3]{chagas2017propertiessteklovaverages}. This chain of calculations lets us conclude that for any fixed $h\in(0,T)$ the sequence of Steklov averages $((u_n)_h)_{n\in\mathbb{N}}$ is uniformly equicontinuous in $C^0([0,T-h];L^1(\Omega))$. 
    For any such fixed  $h\in(0,T)$ and $t\in [0,T-h]$ we look at the sequence of time-slices $(u_n)_h(t)\in L^1(\Omega)$. For $\Phi\in C_c^{\infty}(\bbr\times\bbr)$ with $\|\Phi\|_{L^\infty(\bbr\times\bbr)}\leq 1$, by Fubini's theorem we have
    \begin{align*}
        \int_\Omega (u_n)_h(t) \divG(\Phi)\,dx & = \frac1h\int_\Omega\int_t^{t+h} u_n(s)\,ds\divG(\Phi)\,dx\\
        &= \frac1h \int_t^{t+h} \int_\Omega u_n(s) \divG(\Phi)\,dx\,ds\\
        &\leq \frac1h\int_t^{t+h}\tvG(u_n(s))\,ds\\
        &\leq \frac1h \int_0^T\tvG(u_n(s))\,ds.
    \end{align*}
    By taking supremum over all $\Phi$ with the specified properties, we get the uniform bound
    \begin{align*}
        \sup_{n\in\mathbb{N}} \tvG((u_n)_h(t))\leq \frac1h\int_0^T\tvG(u_n(s))\,ds<\infty.
    \end{align*} Therefore, we get uniform boundedness of the sequence $(u_n)_h(t)$ for $t\in [0,T-h]$ in $\bvG(\bbr)$. Moreover, the Steklov averages $(u_n)_h(t)$ satisfy the boundary condition $(u_n)_h(x,t)=u_0(x)$ for a.e. $x\in\bbr\setminus\Omega$. Hence, we get uniform boundedness of the sequence $(u_n)_h(t)$ in $\bvG(\Omega,u_0)$. Therefore, we may apply the $\bvG$ compactness lemma \cref{lem:com01} to get relative compactness of the sequence $(u_n)_h(t)$ in $L^1(\Omega)$ for any $h\in (0,T)$ and $t\in[0,T-h]$. We can further apply \cite[Lemma 1]{Simon1986} to get that the sequence $(u_n)_h$ is relatively compact in $C^0([0,T-h];L^1(\Omega))$ and hence in $L^1(0,T-h;L^1(\Omega))$. 
    Fix a $S\in(0,T)$. Using Fubini's theorem, for $h\in (0,T-S)$, we have
    \begin{align*}
        \|(u_n)_h-u_n\|_{L^1(0,S;L^1(\Omega))}&=\int_0^{S}\int_{\Omega}\left|\frac1h\int_{t}^{t+h}\int_s^t\de_\tau u_n(\tau)\,d\tau\,ds\right|\,dx\,dt\\
        &\leq \int_0^{S}\int_\Omega\int_t^{t+h}\left|\partial_{\tau}u_n(\tau)\right|\,d\tau\,dx\,\,d\tau\\
        &=\int_0^{S+h}\int_{\max\{o,\tau-h\}}^{\min\{S,\tau\}}\int_\Omega\|\partial_tau u_n(\tau)\|\,dt\,dx\,d\tau\\
        &\leq h\|\de_t u_n\|_{L^1(\Omega_T)}.
    \end{align*} Therefore, it holds that
    \begin{align*}
        \lim_{h\to0}\sup_{n\in\mathbb{N}}\|(u_n)_h-u_n\|_{L^1_{t,x}}=0.
    \end{align*}
    Therefore, $(u_n)h$ converges uniformly in $L^1(0,S;L^1(\Omega))$ to $u_n$ as $h\to0$. Since $(u_n)_h$ is relatively compact in $L^1(0,S;L^1(\Omega))$ so is the sequence $u_n$. In particular, it is relatively compact in $L^1(0,T/2;L^1(\Omega))$.

    The same argument applied to the translated sequence $\tilde{u}_n(t)=u_n(T-t)$ gives us relative compactness in $L^1(0,T/2;L^1(\Omega))$ for $\tilde{u}_n$, which is relative compactness of $u$ in $L^1(T/2,T;L^1(\Omega))$. Therefore, we have the desired compactness.
\end{proof} 
\bibliographystyle{alphaurl}
\bibliography{main}
\end{document}